\documentclass[a4paper]{amsart}
\usepackage{amsmath,amsthm,amsfonts,amssymb,color}
\usepackage{a4wide}
\usepackage[dvipdfmx]{graphicx}
\usepackage{setspace} 

\usepackage{here}
\usepackage{makecell}
\usepackage{enumerate}

\theoremstyle{plain}
\newtheorem{theorem}{Theorem}[section]
\newtheorem{proposition}[theorem]{Proposition}
\newtheorem{corollary}[theorem]{Corollary}
\newtheorem{lemma}[theorem]{Lemma}
\newtheorem{definition}[theorem]{Definition}
\newtheorem{problem}[theorem]{Problem}
\theoremstyle{remark}
\newtheorem{remark}[theorem]{Remark}
\newtheorem{example}[theorem]{Example}

\newcommand{\Harm}{\mathop{\mathrm{Harm}}\nolimits}
\newcommand{\Hom}{\mathop{\mathrm{Hom}}\nolimits}

\date{\today}

\author{Kenji Tanino}
\address{Graduate School of System Informatics, Kobe University \\
1-1 Riokkodai, Nada, Kobe, Hyogo, 657-8501 \\
Japan}
\email{kenji303t@icloud.com}

\author{Tomoki~Tamaru}
\address{Graduate School of System Informatics, Kobe University \\
1-1 Riokkodai, Nada, Kobe, Hyogo, 657-8501 \\
Japan}
\email{249x053x@stu.kobe-u.ac.jp}

\author{Masatake Hirao}
\address{School of Information and Science Technology, Aichi Prefectural University \\
1522-3 Ibaragabasama, Nagakute, Aichi 480-1198 \\
Japan}
\email{hirao@ist.aichi-pu.ac.jp}

\author{Masanori Sawa}
\address{Graduate School of System Informatics, Kobe University \\
1-1 Riokkodai, Nada, Kobe, Hyogo, 657-8501 \\
Japan}
\email{sawa@people.kobe-u.ac.jp}

\thanks{This research is supported in part by Grant-in-Aid for Scientific Research (C) 24K06871 of the Japan Society for the Promotion of Science (JSPS), and by the Early Support Program for Grant-in-Aid for Scientific Research of Kobe University.}

\keywords{Cubature formula, explicit construction, Hilbert identity, integral lattice, simplicial design, spherical design}

\subjclass[2010]{Primary 05E99, 65D32 Secondary 11E76}

\begin{document}
\title{More on the corner-vector construction for spherical designs}

\begin{abstract}
This paper explores a full generalization of the classical corner-vector method for constructing weighted spherical designs, which we call the {\it generalized corner-vector method}.
First we establish a uniform upper bound for the degree of designs obtained from the proposed method.
Our proof is a hybrid argument that employs techniques in analysis and combinatorics, especially a famous result by Xu (1998) on the interrelation between spherical designs and simplicial designs, and the cross-ratio comparison method for Hilbert identities introduced by Nozaki and Sawa~(2013).
We extensively study conditions for the existence of designs obtained from our method, and present many curious examples of degree $7$ through $13$, some of which are, to our surprise, characterized in terms of integral lattices.
\end{abstract}

\maketitle

\section{Introduction} \label{sect:Intro}

Delsarte et al.~\cite{DGS1977} introduces \textit{spherical design} 
as a spherical analogue of balanced incomplete block (BIB) design in applied statistics and combinatorial $t$-design in discrete mathematics.
A finite subset $X$ of a unit sphere $\mathbb{S}^{n-1}$ is defined to be a {\it (weighted) spherical $t$-design}, if for every polynomial of degree at most $t$, the (weighted) average of the function values on $X$ is equal to the integration with respect to the surface measure $\rho$.
Spherical design is closely related to various objects including, but not limited to, spherical cubature in numerical analysis, isometric embeddings of the classical finite-dimensional Banach spaces in functional analysis and optimum experimental designs in applied statistics (see, for example,~\cite{LV1993,X1998} and~\cite{SHY2019}).

One of the most fundamental problems in the spherical design theory is the construction as well as
existence of designs.
Bondarenko et al.~\cite{BRV2013} shows a general existence theorem for designs of small sizes, and gives
an affirmative answer to the Korevaar-Meyers conjecture~\cite{KM1993}.
Although significant progress has been made in the establishment of existence theorems, recent developments of the construction theory have been comparatively modest in scale, even for designs of small degrees.

There are two classical approaches to the construction of spherical designs: {\it product rule} and {\it invariant rule}.
Roughly speaking, the former is a recursive approach that generates a spherical design from designs on simpler spaces, and the latter is an algebraic approach, based on the invariant theory for polynomial spaces, that directly generates a highly symmetric design from group orbits.
A great advantage of the product rule is its simplicity, however, most of the known product rules strongly depend on the assumption of the existence of interval designs for Gegenbauer measure $(1-t^2)^{\lambda-1/2}dt, \lambda>-1/2$ (see, for example,~\cite{Bajnok1992,RB1991,X2022}).
As pointed out by Xiang~\cite{X2022}, the construction of such interval designs is still an open problem.
On the other hand, the invariant rule is a more sophisticated approach, which makes it possible to provide an explicit construction of spherical designs.
We refer the reader to Xiang~\cite{X2022} for discussion on the explicitness of designs.
A major difficulty of the invariant rule is the selection of finitely many points whose group orbits form a design of a given degree.
A traditional and popular criterion is the utilization of {\it corner vectors} for the hyperoctahedral group $B_n$. 

The aim of this paper is to make further progress on the {\it corner-vector method} by considering designs of more general type
\begin{equation}
   \label{eq:invariant0}
   \frac{1}{|\mathbb{S}^{n-1}|} \int_{\mathbb{S}^{n-1}} f(y) d\rho
   =
   \sum_{i=1}^k W_i \sum_{x \in v_{a_i,s_i}^{B_n}} f(x) \
   \text{for every polynomial $f$ of degree at most $t$}.
\end{equation}
We here denote by $v_{a_i,s_i}^{B_n}$ the $B_n$-orbit of  {\it generalized corner vector}
\begin{equation*}
   v_{a_i,s_i} = \frac{1}{\sqrt{a_i^2+s_i}}(a_i,1\ldots,1,0,\ldots,0) \in \mathbb{S}^{n-1}, \quad a_i > 0, \quad s_i \in \{0,1,\ldots,n-1\}.
\end{equation*}
A point $v_{a,s}$ with $a = 1$ is just an $(s+1)$-dimensional \textit{corner vector}, namely the barycentre of 
an $(s+1)$-dimensional face of an $n$-dimensional cross-polytope (see Definition~\ref{def:corner1}).
The corner-vector method is a traditional way of constructing designs with only corner vectors.
While this method can respond to the preference for simplicity of construction, it has the drawback that the generated designs have degree at most $7$ (see Bajnok~\cite{Bajnok2006}).

A classical result on spherical designs is that the vertex set of a regular $(t+1)$-gon is a $t$-design on $\mathbb{S}^1$ (see, for example,~\cite{Hong1982}).
This means that generalized corner vectors can generate designs of any degree in dimension $2$,
since if $t \equiv 0 \pmod{4}$, the vertex set consists of the $B_2$-orbits of $t/4$ generalized corner vectors.
Heo and Xu~\cite[Theorem 2.1]{HX2001} establishes a systematic treatment of the three-dimensional case, and in particular  
shows that the degree of the generated design is uniformly upper bounded by $17$.
Schur~\cite[p.721]{D1923} makes
the first advance for $n \ge 4$, who discovers
a weighted $11$-design on $\mathbb{S}^3$ with a single {\it proper orbit}, namely an orbit $v_{a,s}^{B_4}$ with $a \neq 1$.
The original version of Schur's formula appeared in the context of Waring problem in number theory, as briefly explained in 
Section~\ref{subsect:design}.
Sawa and Xu~\cite{SX2013} discusses
a higher dimensional extension of Schur's formula, and proves
that the degree of any design 
with a single proper orbit is uniformly upper bounded by $11$.
This bound is sharp, since many examples of such $11$-designs have been found in dimensions $3$ through $23$.
Thus a natural question is to ask what happens to designs with two or more
proper orbits.

This paper is organized as follows. Section~\ref{sect:Preliminary} gives preliminaries where we review basic facts on weighted spherical designs, with particular emphasis on the connections to Hilbert identities and simplicial designs.
Sections~\ref{sect:bound} through~\ref{sect:3orbit} are the main body of this paper.
Section~\ref{sect:bound} demonstrates that
for $n \geq 4$,
the degree of any design of type (\ref{eq:invariant0}) is 
uniformly upper bounded by $15$ (Theorem~\ref{thm:main_bound}).
Our proof is a hybrid argument employing techniques in analysis and combinatorics, especially a famous theorem by Xu~\cite{X1998} on the interrelation between spherical and simplicial designs, and 
the \textit{cross-ratio comparison} for
Hilbert identity (see for example~\cite[Theorem 6.6]{NS2013}).
Section~\ref{sect:single} establishes a characterization theorem for $7$-designs of type (\ref{eq:invariant0}) with a single proper orbit (Proposition~\ref{prop:Tanino1}).
As an important corollary of this result, we first show that there exist finitely many such designs including two sporadic examples $v_{2,8}^{B_{16}}$ and $v_{2,11}^{B_{23}}$ (Remark~\ref{thm:Tanino}), and then prove that the degree of all such designs can never be pushed up to $9$ 
(Proposition~\ref{prop:single9design}).
Section~\ref{sect:double} gives a two-orbit version of Proposition~\ref{prop:Tanino1} 
(see Propositions~\ref{prop:Hirao2024-1} and~\ref{prop:Hirao2024-2}), 
and thereby obtain some theoretical results, e.g. a classification of such $9$-designs for $n=3$ (see Table~\ref{tbl:weighted 9-design}).
Section~\ref{sect:3orbit} explores the exhibition of interesting examples of $11$- and $13$-designs with more than $2$ proper orbits.
Finally, Section~\ref{sect:conclusion} is the Conclusion, where some connections between our designs and integral lattices will be also made.

\section{Preliminary}\label{sect:Preliminary}
In this section
we review basic facts concerning spherical designs, with particular emphasis on the connection with Hilbert identity and cubature on simplex, and then quickly overview the background of the {\it corner vector method}.
Some of the materials appearing in this section are available 
in~\cite{SHI2021,Kageyama2019}.

\subsection{Polynomial space}\label{subsect:polynomial}

Let $n$ be a positive integer with $n \ge 2$.
Let $\mathbb{S}^{n-1}$ be the $(n-1)$-dimensional unit sphere centered at the origin, namely, 
\begin{equation*}
\mathbb{S}^{n-1}
= \left\{ (x_1,\ldots,x_n) \in \mathbb{R}^n \mid \sum_{i=1}^n x_i^2=1 \right\}.
\end{equation*}
We denote by $|\mathbb{S}^{n-1}|$ the surface area of $\mathbb{S}^{n-1}$, and by $\rho$ the surface measure on $\mathbb{S}^{n-1}$.
For convention, we write $y^\lambda$ for $y_1^{\lambda_1} \cdots y_n^{\lambda_n}$, where $y = (y_1,\ldots,y_n) \in \mathbb{R}^n$ and $\lambda = (\lambda_1,\ldots,\lambda_n) \in \mathbb{Z}_{\ge 0}^n$.
Then we have
\begin{equation*}
 \frac{1}{|\mathbb{S}^{n-1}|}\int_{\mathbb{S}^{n-1}} y^\lambda d\rho
    = \left\{\begin{array}{cl}
  \frac{ \Gamma(\frac{n}{2}) }{ 2^{ |\lambda|_1 } \Gamma(\frac{ |\lambda|_1 + n }{2}) } \cdot \frac{ \prod_{i=1}^n (\lambda_i)! }{ \prod_{i=1}^n (\frac{\lambda_i}{2})! }    & \text{$\lambda_1 \equiv \cdots \equiv \lambda_n \equiv 0 \pmod{2}$}, \\
0 & \text{otherwise},
\end{array} \right.
\end{equation*}
where $|\lambda|_1 = \lambda_1 + \cdots + \lambda_n$ for $\lambda \in \mathbb{Z}_{\geq 0}^{n}$.

We denote by $\mathcal{P}(\mathbb{R}^n)$ the space of all polynomials with real coefficients on $\mathbb{R}^n$, and by $\mathcal{P}_t(\mathbb{R}^n)$ the subspace of all polynomials of degree at most $t$.
We also denote by $\Hom_t(\mathbb{R}^n)$ or $\Harm_t(\mathbb{R}^n)$, the subspace of $\mathcal{P}_t(\mathbb{R}^n)$ consisting of homogeneous polynomials or harmonic homogeneous polynomials (sometimes called $h$-harmonics) respectively.
In summary,
\begin{equation*}
\begin{split}
\mathcal{P}_t(\mathbb{R}^n) & = \Big\{ \sum_{\substack{ (\lambda_1,\ldots,\lambda_n) \in \mathbb{Z}_{\ge 0}^n \\ \lambda_1+\cdots+\lambda_n \le t}} c_{\lambda_1,\ldots,\lambda_n} x_1^{\lambda_1} \cdots x_n^{\lambda_n} \mid c_{\lambda_1,\ldots,\lambda_n} \in \mathbb{R} \Big\}, \\
    \Hom_t(\mathbb{R}^n)     & = \Big\{ \sum_{\substack{  (\lambda_1,\ldots,\lambda_n) \in \mathbb{Z}_{\ge 0}^n \\ \lambda_1+\cdots+\lambda_n = t}} c_{\lambda_1,\ldots,\lambda_n} x_1^{\lambda_1} \cdots x_n^{\lambda_n} \mid c_{\lambda_1,\ldots,\lambda_n} \in \mathbb{R} \Big\}, \\ 
   \Harm_t(\mathbb{R}^n)     & = \Big\{ f \in \Hom_t (\mathbb{R}^n) \mid \sum_{i=1}^n\frac{\partial^2}{\partial x_i^2} f = 0 \Big\}.
\end{split}
\end{equation*}
It is then obvious that 
It is not entirely obvious but shown
$\mathcal{P}_t(\mathbb{R}^n) = \oplus_{\ell=0}^t \Hom_\ell(\mathbb{R}^n)$.
We use the notation $\mathcal{P}(A)$ for the restriction of $\mathcal{P}(\mathbb{R}^n)$ to $A \subseteq \mathbb{R}^n$, and similarly for $\mathcal{P}(A)$, $\Hom_t(A)$ and $\Harm_t(A)$.

The following proposition is classical in 
spherical harmonics
(see, e.g.,~\cite{EMOT1953,M1966}), which is a basic tool in the study of spherical designs (see, e.g.,~\cite{BB2009,DGS1977}), and is related to the dimension formula for a special subspace of $\mathcal{P}(\mathbb{S}^{n-1})$ (see Proposition~\ref{prop:Molien}).

\begin{proposition}[cf.~\cite{DGS1977,M1966}]\label{prop:decomposition1}
It holds that
\begin{equation}
\label{eq:harm_decomp}
\mathcal{P}(\mathbb{S}^{n-1}) = \bigoplus_{i \ge 0} \Harm_i(\mathbb{S}^{n-1}).
\end{equation}
Moreover (\ref{eq:harm_decomp}) is the orthogonal direct sum with respect to the inner product $(\phi,\psi) = \frac{1}{|\mathbb{S}^{n-1}|} \int_{\mathbb{S}^{n-1}} \phi \psi d\rho$.
\end{proposition}

\subsection{Spherical design, Hilbert identity, simplicial design}\label{subsect:design}

Delsarte et al.~\cite{DGS1977} introduces the notion of spherical design. 

\begin{definition}[Spherical cubature and design]\label{def:spherical_design}
Let $X$ be a finite subset of $\mathbb{S}^{n-1}$, with a weight function $w : X \rightarrow \mathbb{R}_{>0}$. An integration formula of type
\begin{equation}\label{eq:spherical_design1}
\frac{1}{|\mathbb{S}^{n-1}|}\int_{\mathbb{S}^{n-1}} f(y) d\rho
= \sum_{x \in X} w(x)f(x) \ \text{ for every $f \in \mathcal{P}_t(\mathbb{S}^{n-1})$}
\end{equation}
is called a {\it cubature of degree $t$ on $\mathbb{S}^{n-1}$}. 
The pair $(X,w)$ is called 
an 
{\it (equi-weighted) spherical $t$-design} if $w(x)$
takes the constant $1/|X|$, and a {\it weighted spherical $t$-design} otherwise.
\end{definition}

We shall look at some examples.

\begin{example}[cf.~\cite{BB2009,DGS1977}]\label{ex:spherical_dim2}
For any positive integer $k$, let $X$ be the vertex set of a regular $(4k)$-gon inscribed in $\mathbb{S}^1$. An equi-weighted formula of type
\begin{equation*}
\frac{1}{|\mathbb{S}^1|} \int_{\mathbb{S}^1} f(y_1,y_2) d\rho
=
\frac{1}{4k} \sum_{i=0}^{4k-1} f\Big(\cos\frac{(2i+1)\pi}{4k},\sin\frac{(2i+1)\pi}{4k}\Big)
\end{equation*}
is a cubature of degree $4k-1$ on $\mathbb{S}^1$.
\end{example}

\begin{example}\label{ex:spherical_dim3}
An equi-weighted formula of type
\begin{equation}\label{eq:dim3_1}
\begin{gathered}
\frac{1}{|\mathbb{S}^2|} \int_{\mathbb{S}^2} f(y_1,y_2,y_3) d\rho 
=
\frac{1}{6} \sum_{x \in (1,0,0)^{B_3}} f(x)
\end{gathered}
\end{equation}
is a cubature of degree $3$ on $\mathbb{S}^2$.
Here $x^{B_3}$ denotes the orbit of $x \in \mathbb{R}^3$ under the action of the octahedral group $B_3$.
In the present example, we have
\begin{equation}\label{corner_dim3_1}
(1,0,0)^{B_3} = \{ (\pm 1,0,0),(0,\pm 1,0),(0,0,\pm 1)\}.
\end{equation}
This is the vertex set of a regular octahedron in $\mathbb{R}^3$, which is a class of the {\it corner vectors} of the group $B_3$
(see Section~\ref{subsect:corner} for the detail).
An equi-weighted formula of type
\begin{equation*}
\begin{gathered}
\frac{1}{|\mathbb{S}^2|} \int_{\mathbb{S}^2} f(y_1,y_2,y_3) d\rho 
=
\frac{1}{6} \sum_{x \in (\frac{1}{\sqrt{2}},\frac{1}{\sqrt{2}},0)^{B_3}} f(x)
\end{gathered}
\end{equation*}
is also a cubature of degree $3$ on $\mathbb{S}^2$.
In this example, we have
\begin{equation}\label{corner_dim3_2}
\Big(\frac{1}{\sqrt{2}},\frac{1}{\sqrt{2}},0\Big)^{B_3} = \Big\{ \Big(\pm \frac{1}{\sqrt{2}}, \pm\frac{1}{\sqrt{2}},0\Big), \Big(0, \pm \frac{1}{\sqrt{2}}, \pm\frac{1}{\sqrt{2}}\Big),
\Big(\pm \frac{1}{\sqrt{2}}, 0, \pm\frac{1}{\sqrt{2}}\Big) \Big\},
\end{equation}
which is just the set of the midpoints of the edges of a regular octahedron, again a class of the corner vectors of $B_3$.
\end{example}

\begin{example}[Schur's formula]\label{ex:spherical_dim4}
The following is a cubature of degree $11$ 
on $\mathbb{S}^3$:
\begin{equation}\label{eq:Schur}
\begin{gathered}
\frac{1}{|\mathbb{S}^3|} \int_{\mathbb{S}^3} f(y_1,y_2,y_3,y_4) d\rho 
=
\frac{9}{640} \sum_{x \in (\frac{2}{\sqrt{6}},\frac{1}{\sqrt{6}},\frac{1}{\sqrt{6}},0)^{B_4}} f(x) \\
+ \frac{1}{60} \sum_{x \in (1,0,0,0)^{B_4}} f(x) 
+ \frac{1}{96} \sum_{x \in (\frac{1}{\sqrt{2}},\frac{1}{\sqrt{2}},0,0)^{B_4}} f(x)
+ \frac{1}{60} \sum_{x \in (\frac{1}{2},\frac{1}{2},\frac{1}{2},\frac{1}{2})^{B_4}} f(x).
\end{gathered}
\end{equation}
This point configuration is not equi-weighted.
\end{example}

Formula (\ref{eq:Schur}) originally stems from Schur (cf.~\cite[p.721]{D1923}), who finds a certain polynomial identity called {\it Hilbert identity} and thereby makes a significant contribution to Waring problem for $10$th powers of integers.

\begin{definition}\label{def:Hilbert}
A polynomial identity of type
\begin{equation*}
(X_1^2+\cdots+X_n^2)^t
= \sum_{i=1}^N c_i (a_{i1}X_1 + \cdots + a_{in}X_n)^{2t}
\end{equation*}
where $c_i \in \mathbb{R}_{>0}$ and $a_{ij} \in \mathbb{R}$, is called a {\it Hilbert identity}. In particular
this is called a {\it rational (Hilbert) identity} if $c_i \in \mathbb{Q}_{>0}$ and $a_{ij} \in \mathbb{Q}$.
\end{definition}

The following result by Lyubich and Vaserstein~\cite{LV1993} directly relates spherical cubature to Hilbert identity. 
We also refer the reader to \cite{Seidel1994} for a brief explanation of the Lyubich-Vaserstein theorem.

\begin{theorem}[Lyubich-Vaserstein theorem]\label{thm:equivalence_sphere_identity}
Let $n,t$ be positive integers and
\begin{equation*}
c_{n,t}
=  \frac{1}{|\mathbb{S}^{n-1}|}\int_{\mathbb{S}^{n-1}} y_1^{2t} d\rho. 
\end{equation*}
Let $x_i = (x_{i,1},\ldots,x_{i,n}) \in \mathbb{S}^{n-1}$ and $c_i \in \mathbb{R}_{>0}$ for $i=1,\ldots,N$.
The following are equivalent:
\begin{enumerate}
\item[\rm (i)] 
The points $x_i$ and weights $c_i$ give a (weighted) spherical 
design
of index $2t$ on $\mathbb{S}^{n-1}$, namely 
\begin{equation*}
\frac{1}{|\mathbb{S}^{n-1}|}\int_{\mathbb{S}^{n-1}} f(y) d\rho
= \sum_{i=1}^N c_i f(x_i) \ \text{ for every $f \in \Hom_{2t}(\mathbb{S}^{n-1})$};
\end{equation*}
\item[\rm (ii)]
\begin{equation*}
c_{n,t} (X_1^2+\cdots+X_n^2)^t
= \sum_{i=1}^N c_i (x_{i,1}X_1 + \cdots + x_{i,n}X_n)^{2t}.
\end{equation*}
\end{enumerate} 
\end{theorem}

\begin{example}[Schur's identity]\label{ex:spherical_dim4_2}
By Theorem~\ref{thm:equivalence_sphere_identity}, the spherical cubature (\ref{eq:Schur}) is equivalent to the rational Hilbert identity
\begin{equation}\label{eq:Schur1}
\begin{gathered}
 22680 (X_1^2+X_2^2+X_3^2+X_4^2)^5 = \sum_{48} (2X_i \pm X_j \pm X_k)^{10} \\
+ 9 \sum_{4} (2X_i)^{10}
+ 180 \sum_{12} (X_i \pm X_j)^{10}
+ 9 \sum_{8} (X_1 \pm X_2 \pm X_3 \pm X_4)^{10},
\end{gathered}
\end{equation}
where each summation is taken among all combinations of sign changes and permutations of $X_1,X_2,X_3,X_4$.
\end{example}

A spherical cubature is closely connected to a simplicial cubature. We consider the standard orthogonal simplex in $\mathbb{R}^n$, namely
\begin{equation*}
T^n = \{ y=(y_1,\ldots,y_n) \in \mathbb{R}^n \mid 
y_1 \ge 0, 
\ldots, y_n \ge 0,\; | y |_1 \le 1 \},
\end{equation*}
where $|y|_1 = y_1 + \cdots + y_n$.
For $y \in T^n$, we define
\begin{equation*}
W(y_1,\ldots,y_n) = \frac{1}{\sqrt{(1- |y|_1) \Big(\prod_{i=1}^n y_i\Big)}}.
\end{equation*}

The following result can be found in Xu~\cite{X1998}.
\begin{theorem}[\cite{X1998}]\label{thm:simplex}
If
\begin{equation*}
\begin{gathered}
\frac{1}{\int_{T^{n-1}} W(y_1,\ldots,y_{n-1}) \; dy_1\cdots dy_{n-1}} \int_{T^{n-1}} f(y_1,\ldots,y_{n-1}) W(y_1,\ldots,y_{n-1}) \; dy_1\cdots dy_{n-1} \\
= \sum_{i=1}^N c_i f(x_{i,1},\ldots,x_{i,n-1}) \ \text{ for every $f \in \mathcal{P}_t(T^{n-1})$},
\end{gathered}
\end{equation*}
then
\begin{equation*}
\begin{gathered}
\int_{\mathbb{S}^{n-1}} f(y_1,\ldots,y_{n-1},y_{n}) \; d\rho \\
= \sum_{i=1}^N \frac{c_i}{2^{{\rm wt}(x_i)}} \sum_{\pm} f(\pm \sqrt{x_{i,1}}, \ldots, \pm \sqrt{x_{i,n-1}}, \pm \sqrt{1-|x_i|_1}) \ \text{ for every $f \in \mathcal{P}_{2t+1}(\mathbb{S}^{n-1})$}
\end{gathered}
\end{equation*}
where 
 ${\rm wt}(x_i) = |\{ j \mid x_{i, j} \ne 0\}|$ for $i=1,\ldots,N$.
Moreover, the converse direction also holds.
\end{theorem}

\begin{example}[Example~\ref{ex:spherical_dim3}, revisited]\label{ex:sphere_simplex}
By Theorem~\ref{thm:simplex},
the formula (\ref{eq:dim3_1}) of degree $3$ on $\mathbb{S}^2$ is equivalent to a simplicial cubature of type
\begin{equation*}
\begin{gathered}
\frac{1}{\int_{T^2} W(y_1,y_2) dy_1 dy_2} \int_{T^2} f(y_1, y_2) W(y_1, y_2) dy_1  dy_2
= \frac{1}{3} \left ( f(1,0) + f(0,1) + f(0,0) \right ) \\ \ \textit{ for every $f \in \mathcal{P}_1(T^2)$}.
\end{gathered}
\end{equation*}
We note that the 
$6$ vertices
of a regular octahedron~(see (\ref{corner_dim3_1})) 
is transformed to the vertex set of the standard orthogonal simplex $T^2$.
Meanwhile, the formula of type
\begin{equation*}
\begin{gathered}
\frac{1}{|\mathbb{S}^2|} \int_{\mathbb{S}^2} f(y_1,y_2,y_3) d\rho 
=
\frac{1}{6} \sum_{x \in (\frac{1}{\sqrt{2}},\frac{1}{\sqrt{2}},0)^{B_3}} f(x)
\end{gathered}
\end{equation*}
is also a cubature of degree $3$ on $\mathbb{S}^2$, which is equivalent to a simplicial cubature of type
\begin{equation*}
\begin{gathered}
\frac{1}{\int_{T^2} W(y_1,y_2) dy_1 dy_2} \int_{T^2} f(y_1, y_2) W(y_1, y_2) dy_1  dy_2
= \frac{1}{3} \left( f \big (\frac{1}{2},0 \big ) + f \big (0,\frac{1}{2} \big ) + f\big (\frac{1}{2},\frac{1}{2} \big) \right )  \\ \ \textit{ for every $f \in \mathcal{P}_1(T^2)$}.
\end{gathered}
\end{equation*}
Again, we note that the midpoints of the edges of a regular octahedron~(see (\ref{corner_dim3_2})) is transformed to those of the standard orthogonal simplex $T^2$.
Theorem~\ref{thm:simplex} preserves geometric information about corner vectors; more details will be available in Remark~\ref{rem:GCV1}.
\end{example}

\subsection{The corner-vector method}\label{subsect:corner}

A most classical method of constructing spherical designs is the {\it corner-vector method} for the symmetry group $B_n$ of a regular hyperoctahedron in $\mathbb{R}^n$.
Hereafter we write $x^{B_n}$ for the orbit of $x \in \mathbb{R}^n$ under the action of the hyperoctahedral group $B_n$.

\begin{definition}\label{def:corner1}
Let $e_1,\ldots,e_n$ be the standard basis vectors in $\mathbb{R}^n$.
For $k=1,\ldots,n$, let $v_k$ be a vector of type
\begin{equation*}
v_k = \frac{1}{\sqrt{k}} \sum_{i=1}^k e_i \in \mathbb{S}^{n-1}.
\end{equation*}
The {\it corner vectors (for $B_n$)} are the elements of $v_1^{B_n},\ldots,v_n^{B_n}$. 
\end{definition}

\begin{example}[Example~\ref{ex:spherical_dim3}, revisited]\label{ex:corner1}
For $n=2$, the corner vectors are given by
\begin{equation*}
v_1^{B_2} = \{(\pm 1, 0),(0,\pm 1)\}, \quad 
v_2^{B_2} = \Big\{ \Big( \pm \frac{1}{\sqrt{2}}, \pm \frac{1}{\sqrt{2}} \Big)
\Big\},
\end{equation*}
which are the vertices and midpoints of edges of a square in $\mathbb{R}^2$, respectively.
For $n=3$, the corner vectors are given by
\begin{equation*}
\begin{gathered}
v_1^{B_3} = \{ (\pm1,0,0),(0,\pm1,0),(0,0,\pm1) \}, \\
v_2^{B_3} = \Big(\frac{1}{\sqrt{2}},\frac{1}{\sqrt{2}},0\Big)^{B_3} = \Big\{ \Big(\pm \frac{1}{\sqrt{2}}, \pm\frac{1}{\sqrt{2}},0\Big), \Big(0, \pm \frac{1}{\sqrt{2}}, \pm\frac{1}{\sqrt{2}}\Big),
\Big(\pm \frac{1}{\sqrt{2}}, 0, \pm\frac{1}{\sqrt{2}}\Big) \Big\}, \\
v_3^{B_3} = \Big\{ \Big( \pm \frac{1}{\sqrt{3}}, \pm \frac{1}{\sqrt{3}}, \pm \frac{1}{\sqrt{3}} \Big) \Big\}.
\end{gathered}
\end{equation*}
The first and second sets are the vertices and midpoints of edges of a regular octahedron in $\mathbb{R}^3$ 
respectively, and the third is the set of barycenters of faces.
\end{example}

The utility of the corner vector construction is verified from the following more sophisticated result. Below we denote by $\mathcal{P}_t(A)^{B_n}$ the space consisting of all $B_n$-invariant polynomials in $\mathcal{P}_t(A)$, and similarly for ${\rm Harm}_i(A)^{B_n}$.

\begin{theorem}[Sobolev's theorem]\label{thm:Sobolev}
Let $x_1^{B_n},\ldots,x_m^{B_n} \subset \mathbb{S}^{n-1}$.
Let $w: \bigcup_{i=1}^m x_i^{B_n} \rightarrow \mathbb{R}_{>0}$ be a weight function such that 
$w(x) = w(x')$ for every $x,x' \in x_i^{B_n}$ and $i=1,\ldots,m$.
Then the following are equivalent:
\begin{enumerate}
\item[\rm (i)] $(X,w)$ is a (weighted) spherical $t$-design on $\mathbb{S}^{n-1}$;
\item[\rm (ii)] 
\begin{equation*}
\frac{1}{|\mathbb{S}^{n-1}|} \int_{\mathbb{S}^{n-1}} f(y) d\rho 
=
\sum_{x \in X} w(x) f(x) \ \text{ for every $f \in \mathcal{P}_t(\mathbb{S}^{n-1})^{B_n}$}.
\end{equation*}
\end{enumerate}
\end{theorem}

A great advantage of Theorem~\ref{thm:Sobolev} is the reduction of the computational cost in order to check the defining property (\ref{eq:spherical_design1}) of weighted spherical designs.
With Proposition~\ref{prop:decomposition1} in mind, we have to compute the dimension of $\Harm_i(\mathbb{S}^{n-1})^{B_n}$.

\begin{proposition}[cf. Appendix of~\cite{SHI2021}]
\label{prop:Molien}
The harmonic Molien series $\sum_{i=0}^\infty p_i \lambda^i$ where $p_{i} = \dim \Harm_{i}(\mathbb{S}^{n-1})^{B_n}$, is given by
\begin{equation*}
\sum_{i=0}^{\infty} p_i \lambda^{i} =
	 \frac{1}{(1-\lambda^4)(1-\lambda^6) \cdots (1-\lambda^{2n})}.
\end{equation*}
In particular,
\begin{align*}
	\sum_{i=0}^{\infty} p_{i} \lambda^{i} =
    \frac{1}{(1-\lambda^4)(1-\lambda^6) \cdots (1-\lambda^{2n})} = 
	\left\{
	\begin{array}{ll}
		1+\lambda^4+\lambda^6+\lambda^8+\lambda^{10}+ \cdots & \text{if $n=3$}; \\
		1+\lambda^4+\lambda^6+2\lambda^8+\lambda^{10}+ \cdots & \text{if $n=4$}; \\
		1+\lambda^4+\lambda^6+2\lambda^8+2\lambda^{10}+ \cdots & \text{if $n\ge 5$}.
	\end{array} \right.
\end{align*}
\end{proposition}

\begin{corollary}\label{cor:dimension4-8}
It holds that
\begin{equation*}
\begin{gathered}
\dim \Harm_4(\mathbb{S}^{n-1})^{B_n} = \dim \Harm_6(\mathbb{S}^{n-1})^{B_n} = 1 \quad \text{for $n \ge 3$}, \\
\dim \Harm_8(\mathbb{S}^{n-1})^{B_n} =
\left\{\begin{array}{cl}
 1 & \text{for $n = 3$}, \\
 2 & \text{for $n \ge 4$},
\end{array} \right.\\
\dim \Harm_{10}(\mathbb{S}^{n-1})^{B_n} =
\left\{\begin{array}{cl}
 1 & \text{for $n = 3, 4$}, \\
 2 & \text{for $n \ge 5$},
\end{array} \right.\\
\dim \Harm_{2k+1}(\mathbb{S}^{n-1})^{B_n} = 0 \quad \text{for $k=0,1,\ldots$}
\end{gathered}
\end{equation*}
\end{corollary}

Let $\mathcal{S}_n$ be the symmetric group of order $n$. 
It is convenient to use the notation ${\rm sym}(f)$ for a symmetric polynomial as defined by
\begin{equation}\label{eq:sym1}
    {\rm sym} (f) = \frac{1}{|(\mathcal{S}_n)_f|}\sum_{\gamma\in \mathcal{S}_n}f(x_{\gamma(1)},\ldots,x_{\gamma(n)})
\end{equation}
where
\[
(\mathcal{S}_n)_f=\{\gamma\in \mathcal{S}_n \mid f(x_{\gamma(1)},\ldots,x_{\gamma(n)})=f(x_1,\ldots,x_n) \textit{ for all } (x_1,\ldots,x_n)\in\mathbb{R}^n\}.
\]

The following lemmas will be utilized for further arguments in the main body of this paper~(see, for example, Lemma~\ref{lem:Harm81}).

\begin{lemma}[\cite{SHI2021}]\label{lem:basisofHarm}
We define 
homogeneous polynomials 
$f_4$, $f_6$, $f_{8,1}$, $f_{8, 2}$, $f_{10,1}$, $f_{10, 2}$ as
\begin{align*}
      f_4(x) &= {\rm sym}(x_1^4) - \frac{6}{n-1} {\rm sym}(x_1^2 x_2^2) \quad \text{for $n \geq 3$},  \\
      f_6(x) &= {\rm sym}(x_1^6) - \frac{15}{n-1} {\rm sym}(x_1^2 x_2^4) + \frac{180}{(n-1)(n-2)}{\rm sym}(x_1^2x_2^2x_3^2) \quad \text{for $n \geq 3$}, \\
       f_{8,1}(x) &= {\rm sym} (x_1^8)-\frac{28}{n-1}{\rm sym}(x_1^2 x_2^6)+\frac{70}{n-1}{\rm sym}(x_1^4 x_2^4) \quad \text{for $n \geq 3$}, \\
      f_{8,2}(x) &= {\rm sym}(x_1^4 x_2^4)-\frac{6}{n-2}{\rm sym}(x_1^2 x_2^2 x_3^4)
      +\frac{108}{(n-2)(n-3)} {\rm sym}(x_1^2 x_2^2 x_3^2 x_4^2) \quad \text{for $n \geq 4$}, \\
	   f_{10,1}(x) &= 
	 	{\rm sym}(x_{1}^{10}) 
	  	- \frac{45}{n-1} {\rm sym}(x_{1}^{2} x_{2}^{8}) 
		+ \frac{42}{n-1} {\rm sym}(x_{1}^{4} x_{2}^{6}) 
	  	+ \frac{1008}{(n-1)(n-2)} {\rm sym}(x_{1}^{2} x_{2}^{2} x_{3}^{6}) \\ 
	  	 & \qquad - \frac{1260}{(n-1)(n-2)} {\rm sym}(x_{1}^{2} x_{2}^{4} x_{3}^{4}) \quad \text{for $n \geq 3$}, \\
	 f_{10,2}(x) &= 
	 	{\rm sym}(x_{1}^{4} x_{2}^{6}) 
	 	- \frac{6}{n-2} {\rm sym}(x_{1}^{2} x_{2}^{2} x_{3}^{6}) 
	 	- \frac{30}{n-2} {\rm sym}(x_{1}^{2} x_{2}^{4} x_{3}^{4}) 
	 	+ \frac{450}{(n-2)(n-3)} {\rm sym}(x_{1}^{2}x_{2}^{2}x_{3}^{2}x_{4}^{4}) \\
	 	 & \qquad - \frac{10800}{(n-2)(n-3)(n-4)} {\rm sym}(x_{1}^{2} x_{2}^{2} x_{3}^{2} x_{4}^{2} x_{5}^{2})  \quad \text{for $n \geq 5$}.  
\end{align*}
Then $f_4$, $f_6$, $f_{8,1}$, $f_{8, 2}$, $f_{10,1}$, $f_{10, 2}$ are $B_n$-invariant $h$-harmonics.
\end{lemma}
\begin{lemma}\label{lem:vas-FullySymmetricDesign}
   Let 
\begin{align*}
\tilde{f}_{4}(v_{a,s}) &:= (a^2 + s)^2 f_{4}(v_{a,s}), \quad \tilde{f}_{6}(v_{a,s}) := (a^2 + s)^3 f_{6}(v_{a,s}) \\
 \tilde{f}_{8,1}(v_{a,s}) &:= (a^2 + s)^4 f_{8, 1}(v_{a,s}), \quad \tilde{f}_{8, 2}(v_{a,s}) := (a^2 + s)^4 f_{8, 2}(v_{a,s}) \\
  \tilde{f}_{10,1}(v_{a,s}) &:= (a^2 + s)^5 f_{10,1}(v_{a,s}), \quad \tilde{f}_{10,2}(v_{a,s}) 
  := (a^2 + s)^5 f_{10,2}(v_{a,s}). 
\end{align*}
   It holds that 
  \begin{align*}
    \tilde{f}_{4}(v_{a,s}) 
    & = a^4+s -6a^2\frac{s}{n-1}-3\frac{s(s-1)}{n-1}  \quad \text{for $n \geq 3$},\\
    \tilde{f}_{6}(v_{a,s}) 
    &= a^6+s-15(a^4+a^2+s-1)  \frac{s}{n-1}+ 90a^2\frac{s(s-1)}{(n-1)(n-2)} \\
    & \qquad +30 \frac{s(s-1)(s-2)}{(n-1)(n-2)} \quad \text{for $n \geq 3$},                               \\
    \tilde{f}_{8,1}(v_{a,s}) 
    &= a^8+s -28(a^6+a^2+s-1) \frac{s}{n-1}+ 70
    (a^4+\frac{s-1}{2})\frac{s}{n-1} \quad \text{for $n\geq 3$},\\
    \tilde{f}_{8, 2}(v_{a,s}) 
    &= a^4+\frac{s-1}{2} -3(a^4+2a^2+s-2)\frac{s-1}{n-2}+ \frac{9}{2}(4a^2+s-3)\frac{(s-1)(s-2)}{(n-2)(n-3)} \quad \text{for $n\geq4$}, \\
    \tilde{f}_{10,1}(v_{a,s}) 
    &= a^{10}+s-3(15a^8-14a^6-14a^4+15a^2+s-1)\frac{s}{n-1}\\
    & \qquad+126(4a^6 - 10a^4 + 3a^2 - s + 2)\frac{s(s-1)}{(n-1)(n-2)} \quad \text{for $n\geq3$},
    \\
    \tilde{f}_{10,2}(v_{a,s}) 
    &= a^4s + a^6s + s(s - 1)-3(a^6 + 10a^4 + 7a^2 + 6s - 12)\frac{s(s - 1)}{(n-2)}\\
    &\qquad +75(a^4 + 3a^2 + s - 3)\frac{s(s - 1)(s - 2)}{(n-2)(n-3)} \\ 
    &\qquad \qquad -90(5a^2 + s - 4)\frac{s(s - 1)(s - 2)(s - 3)}{(n-2)(n-3)(n-4)} \quad \text{for $n\geq 5$}.
  \end{align*}
\end{lemma}

Now, a crucial demerit of the corner-vector construction is that it cannot 
generate spherical cubature of degree larger than $8$, 
as shown in the following result.

\begin{theorem}[Bajnok's theorem, 
Proposition 15 in~\cite{Bajnok2006}]
\label{thm:Bajnok}
Let $n \ge 3$ be a positive integer.
Assume that $\bigcup_{i=1}^m v_{k_i}^{B_n}$ is a $B_n$-invariant weighted spherical $t$-design. Then it holds that $t \le 7$.
\end{theorem}

To push up the maximum degree of design, we give a generalization of
the notion of corner vectors.

\begin{definition}\label{def:generalized}
Let $e_1,\ldots,e_n$ be the standard basis vectors in $\mathbb{R}^n$.
Let $a$ be a positive real number and $s=0,\ldots,n-1$. 
Let 
\begin{equation*}
   v_{a,s} = \frac{1}{\sqrt{a^2+s}}(ae_1+\sum_{i=1}^s e_{i+1}) \in \mathbb{S}^{n-1}.
\end{equation*}
A {\it generalized corner vector (for $B_n$)} is an element of $v_{a,s}^{B_n}$.
In particular
when $a = 1$,  $v_{1,s}$ exactly coincides with the corner vector $v_{s+1}$.
Moreover, when $a \neq 1$, we refer to $v_{a, s}^{B_n}$ as a \textit{proper orbit}. 
\end{definition}

\begin{example}\label{ex:generalized-corner1}
For $n=2$, take generalized corner vectors of type
\begin{equation*}
        v_{1}^{B_2}=\{(\pm1,0),(0,\pm1)\}, \quad
        v_{\sqrt{3},1}^{B_2}= \Big \{\Big (\pm\frac{1}{2},\pm\frac{\sqrt{3}}{2} \Big),\Big (\pm\frac{\sqrt{3}}{2},\pm\frac{1}{2}\Big) \Big \}.
\end{equation*}
As already seen in Example~\ref{ex:spherical_dim2},
these are the vertices
of a regular dodecagon inscribed in $\mathbb{S}^1$, respectively.
\end{example}

The following explains a geometric interpretation of generalized corner vectors.

\begin{remark}[See, for example, Heo-Xu~\cite{HX2001}]\label{rem:GCV1}
For $n = 3$, we consider the map
\[
\psi: \mathbb{S}^2 \rightarrow T^2, \quad \psi(x_1,x_2,x_3) = (x_1^2,x_2^2).
\]
For a finite subset $X$ of $\mathbb{S}^2$, we write $\psi(X)$ for $\{ \psi(x) \mid x \in X \}$. 
Then $\psi(v_{a,0}^{B_3})$ is the vertex set of $T^2$, namely, 
$\psi(v_{a,0}^{B_3})=\{(0,0),(1,0),(0,1)\}$.
Moreover,
$\psi(v_{a,1}^{B_3})$ is included in the boundary of $T^2$, namely
\[
\begin{gathered}
\psi(v_{a,1}^{B_3}) \subset \{ (x_1,x_2) \in \mathbb{R}^2 \mid x_2 = -x_1+1,\; 0 \le x_1 \le 1\} \cup \{ (x_1,x_2) \in \mathbb{R}^2 \mid x_2 = 0,\; 0 \le x_1 \le 1\} \\
\cup \{ (x_1,x_2) \in \mathbb{R}^2 \mid x_1 = 0,\; 0 \le x_2 \le 1\}.
\end{gathered}
\]
Also, $\psi(v_{a,2}^{B_3})$ is included in the medians of $T^2$, and more precisely
\[
\begin{gathered}
\psi(v_{a,2}^{B_3}) \subset 
\{ (x_1,x_2) \in \mathbb{R}^2 \mid x_2 = x_1,\; 0 \le x_1 \le 1/2\}
\cup \{ (x_1,x_2) \in \mathbb{R}^2 \mid x_2 = -2x_1+1,\; 0 \le x_1 \le 1/2\} \\
\cup \{ (x_1,x_2) \in \mathbb{R}^2 \mid x_1 = -2x_2+1,\; 0 \le x_2 \le 1/2\}.
\end{gathered}
\]
$\psi(v_{1,s}^{B_3})$ coincides with the set of the barycenters of $s$-dimensional faces of $T^2$ 
(see Figure~\ref{fig:GCV1} below).
These observations can also be found in Heo-Xu~\cite{HX2001}, where the situation is explained for regular triangles embedded in $\mathbb{R}^3$. 
A more general treatment of 
Remark~\ref{rem:GCV1}
will be established in Section~\ref{sect:bound}.
\end{remark}
\begin{figure}[htb]
\includegraphics[width=12cm]{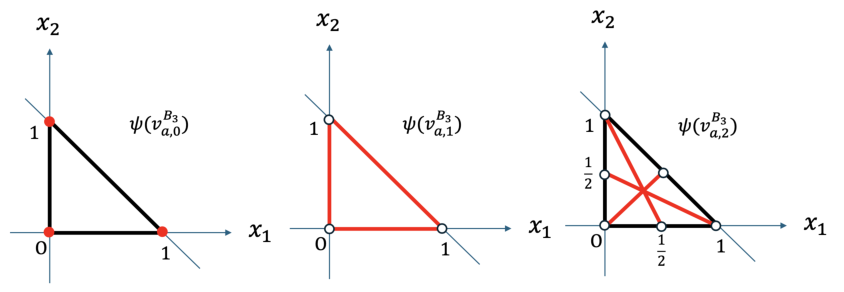}
\caption{$\psi(v_{a, s}^{B_3}) \; (s = 0, 1, 2)$} 
\label{fig:GCV1}
\end{figure}

Hereafter we restrict our attention to the case where $n \ge 3$.
Let $s_1, \ldots, s_k$ be integers with $0 \leq s_i \le n-1$ for all $i$, and let $a_1,\ldots,a_k$ be positive real numbers.
Then, as a generalization of the corner-vector method, we consider a design of type
\begin{equation}
   \label{eq:invariant1}
   \frac{1}{|\mathbb{S}^{n-1}|} \int_{\mathbb{S}^{n-1}} f(y) d\rho
   =
   \sum_{i=1}^k W_i \sum_{x \in v_{a_i,s_i}^{B_n}} f(x)\quad
   \textit{for every } f \in \mathcal{P}_t(\mathbb{S}^{n-1})
\end{equation}

\begin{example}(Example~\ref{ex:spherical_dim4}, revisited).
\label{ex:generalized-corner2}
For $n = 4$, Schur's formula (\ref{eq:Schur1}) can be rewritten in terms of generalized corner vectors as follows:
\begin{equation*}
\begin{gathered}
\frac{1}{|\mathbb{S}^3|} \int_{\mathbb{S}^3} f(y_1,y_2,y_3,y_4) d\rho 
=
\frac{9}{640} \sum_{x \in  v_{2,2}^{B_4}} f(x) 
+ \frac{1}{60} \sum_{x \in v_{1}^{B_4}} f(x) 
+ \frac{1}{96} \sum_{x \in v_{2}^{B_4}} f(x)
+ \frac{1}{60} \sum_{x \in v_{4}^{B_4}} f(x).
\end{gathered}
\end{equation*}
This can be checked by applying Theorem~\ref{thm:Sobolev} to $ f_4,f_6,f_{8,1},f_{8,2}$ and $f_{10,1}$.
\end{example}

What is remarkable here is that Schur's formula has degree $11$, contrary to the situation of Bajnok's theorem.
More generally, Sawa and Xu~\cite{SX2013} establishes
that the maximum degree of a $B_n$-invariant design with 
only one proper orbit
$v_{a,s}^{B_n}$ can be pushed up to $11$.
Then what about designs with 
two or more proper orbits
$v_{a_1,s_1}^{B_n},v_{a_2,s_2}^{B_n},\ldots$? 
This is the main subject of this paper, although in Section~\ref{sect:single}, some new results are established for designs with a single proper orbit $v_{a,s}^{B_n}$.

\section{A uniform upper bound for the degree of our designs}\label{sect:bound}

Let $n, s_i$ be integers with $n \geq 3$  and $0\le s_i \le n-1$, and let $a_i>0$.
In this section
we prove an upper bound for the degree of a weighted spherical design of type (\ref{eq:invariant1}). The presentation in this section is conscious of algebraic and geometric interpretations of generalized corner vectors.
  
\subsection{Uniform bound} \label{subsect:bound_main}

\begin{theorem}\label{thm:main_bound}
Let $n\geq4$. Suppose that 
\begin{equation}
   \label{eq:invariant2}
   \frac{1}{|\mathbb{S}^{n-1}|} \int_{\mathbb{S}^{n-1}} f(y) d\rho
   =
   \sum_{i=1}^k W_i \sum_{x \in v_{a_i,s_i}^{B_n}} f(x)\quad
   \textit{for every } 
   f \in \mathcal{P}_t(\mathbb{S}^{n-1}).
\end{equation}
Then it holds that $t \le 15$.
\end{theorem}

\begin{remark}\label{rem:main_bound}
As shown by Heo and Xu~\cite[Theorem~2.1]{HX2001}, 
the maximum degree of 
weighted spherical designs on $\mathbb{S}^2$
of type~(\ref{eq:invariant2}) is upper bounded by $17$.
\end{remark}

Since designs of type (\ref{eq:invariant2}) are centrally symmetric, it suffices to take care of even degrees $2,4,\ldots,2\lfloor t/2 \rfloor$.
As a corollary of Theorems~\ref{thm:main_bound} and~\ref{thm:equivalence_sphere_identity}, we obtain the following result.

\begin{corollary}\label{cor:main_bound}
Suppose that
\begin{equation}
\label{eq:invariant3}
c_{n,r} (X_1^2+\cdots+X_n^2)^r
= \sum_{i=1}^k \lambda_i \sum_{\sigma, \pm} (a_i X_{\sigma(1)} \pm X_{\sigma(2)} \pm \cdots \pm X_{\sigma(s_{i}+1)})^{2r}
\end{equation}
where the second summation is taken over all permutations 
$\sigma \in \mathcal{S}_n$ 
and all sign changes $\pm$.
Then it holds that $r \le 7$.
\end{corollary}

The proof of Theorem~\ref{thm:main_bound}
is substantially divided into three parts. The first part is an application of the {\it cross-ratio comparison} that was first introduced by Nozaki and Sawa~\cite[Theorem 6.6]{NS2013}. The remaining two parts involve elementary calculus on $B_n$-invariant $h$-harmonics
of degree $8$ and Xu's characterization theorem (see Theorem~\ref{thm:simplex}), respectively.

\subsection{Proof of Theorem~\ref{thm:main_bound}}\label{subsect:proof_bound}

As briefly mentioned in Section~\ref{subsect:bound_main}, 
we need three preliminary lemmas, each including a key idea that can also be applied in the study of design
of a different type than~(\ref{eq:invariant2}).

\begin{lemma}\label{lem:CRC}
Let $n \ge 4$.
Suppose that there exists a weighted $t$-design of type (\ref{eq:invariant2}) with $s_i \ge 3$ for some $i$. Then it holds that $t \le 15$.
\end{lemma}
\begin{proof}
The proof is based on the cross-ratio comparison for monomials 
$X_1^2 X_2^2 X_3^6 X_4^6$, $X_1^2 X_2^4 X_3^4 X_4^6$ and $X_1^4 X_2^4 X_3^4 X_4^4$.

First we compare the ratio of the coefficients of 
$X_1^2 X_2^2 X_3^6 X_4^6$ and $X_1^2 X_2^4 X_3^4 X_4^6$ 
on the both sides of (\ref{eq:invariant3}).
On the left side, we have
\[
\binom{8}{1} \binom{7}{3} \binom{4}{1} : \binom{8}{1} \binom{7}{3} \binom{4}{2}
= \binom{4}{1} : \binom{4}{2} = 2:3.
\]
Whereas, on the right side, we have
\[
\begin{gathered}
\binom{16}{2} \binom{14}{6} \binom{8}{2} \sum_i A_i (2a_i^2 + 2a_i^6 + s_i -3) : \binom{16}{2} \binom{14}{6} \binom{8}{4}  \sum_i A_i (a_i^2 + 2a_i^4 + a_i^6 + s_i -3) \\
= 2 \sum_i A_i (2a_i^2 + 2a_i^6 + s_i -3) : 5 \sum_i A_i (a_i^2 + 2a_i^4 + a_i^6 + s_i -3)
\end{gathered}
\]
where 
$A_i = \lambda_i 2^{s_i} \binom{n-4}{s_i-3}$
for $i=1,\ldots,k$.
Then
\[
3 \sum_i A_i (2a_i^2 + 2a_i^6 + s_i -3) = 5 \sum_i A_i (a_i^2 + 2a_i^4 + a_i^6 + s_i -3)
\]
and therefore
\begin{equation}\label{eq:CRC1}
\sum_i A_i (a_i^6 -10a_i^4 + a_i^2) = 2 \sum_i A_i (s_i -3).
\end{equation}

Next we compare the ratio of the coefficients of 
$X_1^2 X_2^4 X_3^4 X_4^6$ and $X_1^4 X_2^4 X_3^4 X_4^4$
on the both sides of (\ref{eq:invariant3}).
On the left side, we have
\[
\binom{8}{2} \binom{6}{2} \binom{4}{3} : \binom{8}{2} \binom{6}{2} \binom{4}{2}
= \binom{4}{3} : \binom{4}{2} = 2:3.
\]
Whereas on the right side, we have
\[
\begin{gathered}
\binom{16}{4} \binom{12}{4} \binom{8}{6} \sum_i A_i (a_i^2 + 2a_i^4 + a_i^6 + s_i -3) : \binom{16}{4} \binom{12}{4} \binom{8}{4}  \sum_i A_i (4a_i^4 + s_i -3) \\
2 \sum_i A_i (a_i^2 + 2a_i^4 + a_i^6 + s_i -3) : 5 \sum_i A_i (4a_i^4 + s_i -3).
\end{gathered}
\]
Then
\[
5 \sum_i A_i (4a_i^4 + s_i -3) = 3 \sum_i A_i (a_i^2 + 2a_i^4 + a_i^6 + s_i -3)
\]
and therefore
\begin{equation}\label{eq:CRC2}
\sum_i A_i (3a_i^6 -14a_i^4 + 3a_i^2) = 2 \sum_i A_i (s_i -3).
\end{equation}

In summary, by subtracting (\ref{eq:CRC2}) from (\ref{eq:CRC1}), we have
\begin{equation*}
0 = \sum_i A_i (a_i^6 -2a_i^4 + a_i^2) = \sum_i A_i a_i^2 (a_i^2 - 1)^2,
\end{equation*}
which implies 
$a_i \in \{0, 1\}$.
This is a contradiction to Bajnok's theorem (see Theorem~\ref{thm:Bajnok}).
\end{proof}

\begin{lemma}\label{lem:Harm81}
Let $n \ge 8$.
Suppose that there exists a weighted $t$-design of type (\ref{eq:invariant2}) with $s_i \in \{0,1,2\}$ for all $i$. Then it holds that $t \le 7$.
\end{lemma}
\begin{proof}
By combining (\ref{eq:invariant2}) with Lemma~\ref{lem:vas-FullySymmetricDesign}
for $\tilde{f}_{8,1}$, we have
\begin{equation*}
0 = \sum_i \tilde{W}_i \Big(a_i^8 + s_i + 7(-4a_i^6 + 10a_i^4 - 4a_i^2 + s_i -1) \frac{s_i}{n-1} \Big), 
\end{equation*}
where $\tilde{W}_i = W_i |v_{a_i, s_i}^{B_n}| > 0$.
For $i = 1,\ldots,k$, let
\[
A_{a_i,s_i} = a_i^8 + s_i + 7(-4a_i^6 + 10a_i^4 - 4a_i^2 + s_i -1) \frac{s_i}{n-1}.
\]
Since $\tilde{W}_i > 0$
for every $i$, there exists $j$ such that
$A_{a_j,s_j} < 0$.
Clearly $s_j \ne 0$.
It thus follows from elementary calculus that
\[
n-1 <
\left\{ \begin{array}{ll}
\displaystyle \frac{28a_j^6-70a_j^4+28a_j^2}{a_j^8+1} < 4 & \ \text{ if $s_j = 1$}; \\
\displaystyle \frac{56a_j^6-140a_j^4+56a_j^2-14}{a_j^8+2} < 7 & \ \text{ if $s_j = 2$}.
\end{array} \right.
\]
\end{proof}

\begin{lemma}\label{lem:simplex_proof}
Let $n \ge 3$.
Suppose that there exists a weighted $t$-design of type (\ref{eq:invariant2}) with $s_i \le n-2$ for all $i$. Then it holds that $t \le 2n-1$.
\end{lemma}
\begin{proof}
Let $n \ge 3$, and suppose that there exists a weighted $(2n+1)$-design of type (\ref{eq:invariant2}) with $s_i \le n-2$ for all $i$.
As in Remark~\ref{rem:GCV1}, we consider the map
\[
\psi: \mathbb{S}^{n-1} \rightarrow T^{n-1}, \quad \psi(x_1,\ldots,x_{n-1},x_n) = (x_1^2,\ldots,x_{n-1}^2).
\]
Then by Theorem~\ref{thm:simplex}, the set
\[
S = \Big\{ \psi(x) \in T^{n-1} \mid x \in \bigcup_{i=1}^k v_{a_i,s_i}^{B_n} \Big\}
\]
is a weighted simplicial 
design
of type
\begin{equation*}
\sum_{s \in S} c_s f(z_s)
=
\frac{1}{ \int_{ T^{n-1} } W(y) dy_1 \cdots dy_{n-1} } \int_{ T^{n-1} } f(y) W(y) dy_1\cdots dy_{n-1} \ \text{for every $f \in \mathcal{P}_n(T^{n-1})$}.
\end{equation*}

Let $x = (x_1,\ldots,x_n) \in \bigcup_{i=1}^k v_{a_i,s_i}^{B_n}$.
Since $s_i \le n-2$ for all $i$, there exists at least one nonzero coordinate of $x$.
If $x_n = 0$, then $\psi(x)$ is included in the affine hyperplane $\pi_0: x_1+\cdots+x_n=1$.
If $x_i = 0$ where $i=1,\ldots,n-1$, then $\psi(x)$ is included in the hyperplane $\pi_i: x_i = 0$.
Then the polynomial
\[
f(x_1,\ldots,x_n) = \Big( \prod_{i=1}^{n-1} x_i \Big) \Big(1 - \sum_{j=1}^n x_j \Big)
\]
has degree $n$, which vanishes at the boundary of $T^{n-1}$ but takes positive values at the interior of $T^{n-1}$.
Hence it follows that
\[
0
< \frac{1}{ \int_{ T^{n-1} } W(y) dy_1\cdots dy_{n-1} } \int_{ T^{n-1} } f(y) W(y) dy_1\cdots dy_{n-1} 
= \sum_{s \in S} c_s f(z_s)
= 0,
\]
which is a contradiction.
\end{proof}

We are now in a position to complete the proof of Theorem~\ref{thm:main_bound}

\begin{proof}[Proof of Theorem~\ref{thm:main_bound}]
The result is proved by Lemma~\ref{lem:CRC} for 
$s_i \geq 3$
for some $i$, and by Lemmas~\ref{lem:Harm81},~\ref{lem:simplex_proof} for 
$s_i < 3$ 
for all $i$.~(see Table~\ref{tbl:upper bound})
\end{proof}

\begin{table}[htb]
\begin{center}
\caption{Upper bound for the maximum degree $t_n$ of 
weighted designs on $\mathbb{S}^{n-1}$ with generalized corner vectors}
\label{tbl:upper bound}
\begin{tabular}{ccccccc}
\hline
$n$                  & 4          & 5 & 6 & 7 & 8 & 9 \\ \hline
\makecell[l]{$\exists i \; (s_i\geq 3)$\\ Lemma~\ref{lem:CRC}}      
& $\boldsymbol{t_n} \mathbf{\leq 15}$ & $\boldsymbol{t_n} \mathbf{\leq 15}$ & $\boldsymbol{t_n} \mathbf{\leq 15}$  & $\boldsymbol{t_n} \mathbf{\leq 15}$  & $\boldsymbol{t_n} \mathbf{\leq 15}$  & $\boldsymbol{t_n} \mathbf{\leq 15}$  \\ \hline
\makecell[l]{$\forall i \; (s_i=0,1,2)$ \\Lemma~\ref{lem:Harm81}}
& \makecell[c]{-} & \makecell[c]{-} & \makecell[c]{-} & \makecell[c]{-}  & $\boldsymbol{t_n} \mathbf{\leq 7}$  &  $\boldsymbol{t_n} \mathbf{\leq 7}$   \\ \hline
\makecell[l]{$\forall i \; (s_i=0,1,2)$ \\Lemma~\ref{lem:simplex_proof}}
 & $\boldsymbol{t_n} \mathbf{\leq 7}$ & $\boldsymbol{t_n} \mathbf{\leq 9}$ & $\boldsymbol{t_n} \mathbf{\leq 11}$ & $\boldsymbol{t_n} \mathbf{\leq 13}$  & $t_n\leq15$  &  $t_n\leq17$    \\ \hline
\end{tabular}
\end{center}
\end{table}

\section{Designs with a single proper orbit}\label{sect:single}
Throughout this section we assume $n \ge 3$ and consider a 
design
of type
\begin{equation}
   \label{eq:invariant1orbit}
   \frac{1}{|\mathbb{S}^{n-1}|} \int_{\mathbb{S}^{n-1}} f(y) d\rho
   =
   \frac{1}{|v_{a,s}^{B_n}|} \sum_{x \in v_{a,s}^{B_n}} f(x)\quad
   \textit{for every } 
   f \in \mathcal{P}_t(\mathbb{S}^{n-1}).
\end{equation}

\begin{proposition}\label{prop:Tanino1}
      Let $n \geq 3$.
      A $7$-design of type (\ref{eq:invariant1orbit}) exists if and only if,
          \begin{gather}
              a^2 = \frac{3s\pm\sqrt{(2+n)s(1-n+3s)}}{n-1},\label{eq:singleorbit-1}\\
            h_{\pm}(n, s) = n^3 + (2-9s)n^2+(-7-9s+12s^2)n+6s^2+18s+4 \notag \\
            \qquad \qquad \qquad \pm(n^2-3(-2+s)n-3s-7)\sqrt{(2+n)s(1-n+3s)} = 0. 
            \label{eq:singleorbit-2}
          \end{gather}
\end{proposition}

\begin{proof}
    By Theorem~\ref{thm:Sobolev}, Corollary~\ref{cor:dimension4-8} and Lemma~\ref{lem:vas-FullySymmetricDesign},
    a $7$-design of type (\ref{eq:invariant1orbit}) exists if and only if 
    $\tilde{f}_{4}(v_{a,s}) = \tilde{f}_{6}(v_{a,s})=0$.

    Solving $\tilde{f}_{4}(v_{a,s})=0$ for $a^2$, we obtain (\ref{eq:singleorbit-1}).
    Substituting (\ref{eq:singleorbit-1}) into $\tilde{f}_{6}(v_{a,s})=0$, we also obtain (\ref{eq:singleorbit-2}).
\end{proof}

\begin{remark}\label{rem:Tanino}
    When $s=0$ in (\ref{eq:singleorbit-2}), we have
    \begin{equation*}
        0=n^3+2n^2-7n+4 = (n - 1)^2 (n + 4),
    \end{equation*}
    whose roots are $1$ and $-4$, contradicting the assumption that $n\geq3$.
\end{remark}

\begin{remark}\label{cor:Tanino}
    Let $n \geq 3$.
    If a weighted $7$-design of type (\ref{eq:invariant1orbit}) exists, then 
    $n^2-3(-2+s)n-3s - 7=0$
    or
    $(2+n)s(1-n+3s)$ is a square number.
    But 
    $n^2-3(-2+s)n-3s - 7=0$
    does not hold.
Suppose contrary.
Since $n$ and $s$ are integers, the discriminant of 
    $n^2-3(-2+s)n-3s - 7=0$
is a square, that is, there exists an integer $m$ such that
\begin{equation}\label{eq:singleorbit-4}
    9s^2-24s+64=m^2.
\end{equation}
Then, the solutions of (\ref{eq:singleorbit-4}) are $(s,m)=(0,\pm8),(1,\pm7),(5,\pm13)$, 
for which $n^2-3(-2+s)n-3s - 7=0$ has no integer solutions.
\end{remark}

\begin{theorem}
   \label{thm:main2}
   There exist only
   finitely many pairs $(n,s)$ that satisfy the condition (\ref{eq:singleorbit-2}).
\end{theorem}

\begin{proof}
Substituting $n = x$, $s = y$ into (\ref{eq:singleorbit-2}), we have
\begin{align*}
h_+(x, y) h_-(x, y) &=\{ n^3 + (2 - 9 s) n^2 + (-7 - 9 s + 12 s^2) n + 6 s^2 + 18 s + 4 \}^2 \\
& \qquad-\{ n^2 - 3 (-2 + s) n - 3 s - 7) \sqrt{(n + 2) s (1 - n + 3 s)}\}^2 \\
&= (-1 + x)^2 (1 + y) \big \{ x^4 + (6-9y)x^3 + (27y^2-30y+1)x^2 \\
& \qquad - (27y^3-54y^2-9y+24)x - 18y^3-36y^2+30y+16 \big \} = 0.
\end{align*}
By noting $x\geq 3$ and $y\geq 0$, we obtain
   \begin{equation}
      \label{eq:curve1}
      \begin{gathered}
         f(x,y) := x^4 + (6-9y)x^3 + (27y^2-30y+1)x^2 - (27y^3-54y^2-9y+24)x \\
         \qquad - 18y^3-36y^2+30y+16 = 0.
      \end{gathered}
   \end{equation}
   Suppose $f_y(x,y) = 0$. Since
   \begin{eqnarray*}
      f_y(x,y)
      &=& -9 x^3 + 6 x^2 (9 y - 5) + x (-81 y^2 + 108 y + 9) - 54 y^2 - 72 y + 30 \\
      &=& -3 (x - 3 y + 1) (3 x^2 - 9 x y + 7 x - 6 y - 10),
   \end{eqnarray*}
   we have
   \begin{equation}
      \label{eq:partial1}
      y=\frac{x+1}{3} \quad \text{ or } \quad y = \frac{-10 + 7 x + 3 x^2}{3 (2 + 3 x)} \ \text{ and } \ x \ne -\frac{2}{3}.
   \end{equation}
   Substituting (\ref{eq:partial1}) into
   \[
      f_x(x,y) = 4 x^3 - 9 x^2 (3 y - 2) + x (54 y^2 - 60 y + 2) - 27 y^3 + 54 y^2 + 9 y - 24 = 0,
   \]
   we obtain $(x,y) = (2,1)$ or $(1,0)$.
   Since $f(2,1) =  0$ and $f(1,0) = 0$, these are the only singular points of the curve $C$ defined by (\ref{eq:curve1}).
   It is not entirely obvious but shown that $(2,1)$ or $(1,0)$ are ordinary multiple points of multiplicity $2$, and therefore $C$ has genus one. By Siegel's theorem~\cite{Siegel1929}, there exist finitely many pairs $(n,s)$ for which (\ref{eq:singleorbit-2}) holds.
\end{proof}

\begin{remark}\label{thm:Tanino}
     By using Mathematica, we find that a weighted $7$-design of type (\ref{eq:invariant1orbit}) exists if $(n,a,s)=(16,2,8),(23,2,11)$.
\end{remark}

In Section~\ref{sect:conclusion}, we characterize 
$v_{2,8}^{B_{16}}$ and $v_{2,11}^{B_{23}}$
in terms of shells of integral lattices.

\bigskip

Contrary to the situation of $7$-designs, there do not exist $9$-designs, as shown in the following result.

\begin{proposition}\label{prop:single9design}
    Let $n \geq 3$.
    A weighted $9$-design of type (\ref{eq:invariant1orbit}) does not exist.
\end{proposition}
\begin{proof}
By Remark \ref{rem:Tanino}, we may assume $s\geq1$.
If a weighted $9$-design of type (\ref{eq:invariant1orbit})
exists, then by Theorem~\ref{thm:Sobolev}, 
Corollary~\ref{cor:dimension4-8} and Lemma~\ref{lem:vas-FullySymmetricDesign}, we have
$\tilde{f}_{4}(v_{a, s}) = \tilde{f}_{6}(v_{a, s}) = \tilde{f}_{8, 1}(v_{a, s}) = 0$.
By computing a Groebner basis for the ideal $\langle \tilde{f}_4, \tilde{f}_6, \tilde{f}_{8,1} \rangle$,
   we have
   \begin{align*}
      -277504 - 32752 n - 1412 n^2 - 200 n^3 + 3 n^4 = 0.
   \end{align*}
However
this equation has no integer solutions.
\end{proof}

\section{Designs with two proper orbits}\label{sect:double}

Throughout this section we assume that $n \ge 3$.
Let $a_1,a_2 > 0$, and let $s_1,s_2$ be integers 
with 
$0 \le s_i \le n-1$.\ We consider a weighted 
design
of type
\begin{equation}
   \label{eq:invariant2orbit}
   \frac{1}{|\mathbb{S}^{n-1}|} \int_{\mathbb{S}^{n-1}} f(y) d\rho
   =
   W_1 \sum_{x \in v_{a_1,s_1}^{B_n}} f(x) + W_2 \sum_{x \in v_{a_2,s_2}^{B_n}} f(x) \quad
   \textit{for every } f \in \mathcal{P}_t(\mathbb{S}^{n-1}).
\end{equation}

Let 
      \begin{align*}
         \begin{cases}
            G_{4, 6} (a_1, a_2, s_1, s_2) 
            &=  f_{4}(v_{a_1, s_1})  f_{6} (v_{a_2, s_2}) - f_{6}(v_{a_1, s_1}) f_{4} (v_{a_2, s_2}),                  \\
            G_{4, 8,i} (a_1, a_2, s_1, s_2)
            &= f_{4}(v_{a_1, s_1}) f_{8,i}(v_{a_2, s_2}) -  f_{8,i} (v_{a_1, s_1}) f_4(v_{a_2, s_2}), \quad i = 1, 2,  \\
            G_{6, 8,i} (a_1, a_2, s_1, s_2)
            &= f_{6}(v_{a_1, s_1}) f_{8,i}(v_{a_2, s_2}) -  f_{8,i} (v_{a_1, s_1}) f_6(v_{a_2, s_2}), \quad i = 1, 2,  \\
            G_{8,1, 8, 2} (a_1, a_2, s_1, s_2) 
            &= f_{8,1}(v_{a_1, s_1}) f_{8,2}(v_{a_2, s_2}) -  f_{8,2} (v_{a_1, s_1}) f_{8,1}(v_{a_2, s_2}).
        \end{cases}
      \end{align*}

\begin{remark}
By substituting $a_1 = a_2 = 1$, $s_1 = k_1 - 1$ and $s_2 = k_2 - 1$ into $G_{4,6}(a_1, a_2, s_1, s_2)$, it can be confirmed that this function coincides, up to a constant multiple, with the function $G(k_1, k_2)$ defined in Bajnok~\cite[p.390]{Bajnok2006}.
\end{remark}

\subsection{Characterization of $7$-designs}\label{subsect:2orb_degr7}

In this subsection, we show a
characterization theorem for $7$-designs with two orbits. 

\begin{proposition}
   \label{prop:Hirao2024-1}
    Let $n \geq 3$.
    A weighted $7$-design of type (\ref{eq:invariant2orbit}) exists if and only if one of the following cases (i)-(iii) holds:

   \begin{enumerate}
      \item[(i)]
           \[
               f_4(v_{a_1,s_1})  f_4(v_{a_2,s_2}) < 0
               \;,\; G_{4,6}(a_1,a_2,s_1,s_2) =0;
            \]

      \item[(ii)]
           \[
               f_4(v_{a_1,s_1}) = f_4(v_{a_2,s_2}) = 0 \;,\;
               f_6(v_{a_1,s_1}) f_6(v_{a_2,s_2}) < 0;
            \]
      \item[(iii)]
            \begin{equation}
               \label{eq:zeros1}
               f_4(v_{a_1,s_1}) = f_4(v_{a_2,s_2}) = f_6(v_{a_1,s_1}) = f_6(v_{a_2,s_2}) = 0.
            \end{equation}
   \end{enumerate}
\end{proposition}

\begin{proof}
We solve the following system of liner equations
\begin{equation}
    \begin{bmatrix}
        1                   & 1                   \\
        f_{4}(v_{a_1, s_1}) & f_{4}(v_{a_2, s_2}) \\
        f_{6}(v_{a_1, s_1}) & f_{6}(v_{a_2, s_2})
    \end{bmatrix}
    \begin{bmatrix}
        \tilde{W}_1 \\ \tilde{W}_2
    \end{bmatrix}
    = \begin{bmatrix}
        1 \\ 0 \\ 0
    \end{bmatrix},
    \label{eq:system1}
\end{equation}
where $\tilde{W}_i = W_i|v_{a_i, s_i}^{B_n}| > 0,\;  (i = 1, 2)$.
We divide the situation into three cases:
(a) $f_{4}(v_{a_1, s_1}) \neq f_{4}(v_{a_2, s_2})$,\ (b) $f_{4}(v_{a_1, s_1}) = f_{4}(v_{a_2, s_2}), \;f_{6}(v_{a_1, s_1}) \neq  f_{6}(v_{a_2, s_2})$ and (c) $f_{4}(v_{a_1, s_1}) = f_{4}(v_{a_2, s_2}), \;f_{6}(v_{a_1, s_1}) =  f_{6}(v_{a_2, s_2})$.

We first consider Case (a).
Applying the row reduction in the standard linear algebra to the augmented coefficient matrix in (\ref{eq:system1}),
we have
\[
    \begin{bmatrix}
        1                   & 1                   & 1 \\
        f_{4}(v_{a_1, s_1}) & f_{4}(v_{a_2, s_2}) & 0 \\
        f_{6}(v_{a_1, s_1}) & f_{6}(v_{a_2, s_2}) & 0
    \end{bmatrix}
    \rightarrow
    \begin{bmatrix}
        1 & 1                                         & 1                     \\
        0 & f_{4}(v_{a_2, s_2}) - f_{4}(v_{a_1, s_1}) & - f_{4}(v_{a_1, s_1}) \\
        0 & f_{6}(v_{a_2, s_2}) - f_{6}(v_{a_1, s_1}) & - f_{6}(v_{a_1, s_1}) \\
    \end{bmatrix}
    \rightarrow
    \begin{bmatrix}
        1 & 0 & \frac{f_{4}(v_{a_2, s_2})}{f_{4}(v_{a_2, s_2}) - f_{4}(v_{a_1, s_1}) }       \\
        0 & 1 & - \frac{f_{4}(v_{a_1, s_1})}{f_{4}(v_{a_2, s_2}) - f_{4}(v_{a_1, s_1}) }     \\
        0 & 0 & \frac{G_{4,6}(a_1, a_2, s_1, s_2)}{f_{4}(v_{a_2, s_2}) - f_{4}(v_{a_1, s_1}) }
    \end{bmatrix}.
\]
Thus in this case, a weighted 7-design of type (\ref{eq:invariant2orbit}) exists if and only if
\begin{equation*}
    \begin{cases}
        f_{4}(v_{a_1, s_1}) \neq f_{4}(v_{a_2, s_2}),                                              \\
         \frac{f_{4}(v_{a_2, s_2})}{f_{4}(v_{a_2, s_2}) - f_{4}(v_{a_1, s_1}) } > 0,     \\
        - \frac{f_{4}(v_{a_1, s_1})}{f_{4}(v_{a_2, s_2}) - f_{4}(v_{a_1, s_1}) }  > 0, \\
        G_{4, 6} (a_1, a_2, s_1, s_2) = 0.
    \end{cases}
\end{equation*}
This is equivalent to
\begin{equation}
\label{eq:result1}
    \begin{cases}
        f_{4}(v_{a_1, s_1})  f_{4}(v_{a_2, s_2}) < 0, \\
        G_{4, 6} (a_1, a_2, s_1, s_2) = 0.
    \end{cases}
\end{equation}

Next we consider Case (b).
Applying the row reduction arguments again, 
the augmented coefficient matrix can be reduced to
\[
    \begin{bmatrix}
        1 & 0 & \frac{f_{6}(v_{a_2, s_2})}{f_{6}(v_{a_2, s_2}) - f_{6}(v_{a_1, s_1})}   \\
        0 & 1 & - \frac{f_{6}(v_{a_1, s_1})}{f_{6}(v_{a_2, s_2}) - f_{6}(v_{a_1, s_1})} \\
        0 & 0 & - f_{4}(v_{a_1, s_1})                                                   \\
    \end{bmatrix}.
\]
Thus in this case, a weighted 7-design of type (\ref{eq:invariant2orbit}) exists if and only if
\begin{equation*}
    \begin{cases}
        f_{4}(v_{a_1, s_1}) = f_{4}(v_{a_2, s_2})  = 0,                                         \\
        f_{6}(v_{a_2, s_2}) \neq  f_{6}(v_{a_1, s_1})  ,                                        \\
         \frac{f_{6}(v_{a_2, s_2})}{f_{6}(v_{a_2, s_2}) - f_{6}(v_{a_1, s_1})}  > 0, \\
         - \frac{f_{6}(v_{a_1, s_1})}{f_{6}(v_{a_2, s_2}) - f_{6}(v_{a_1, s_1})} > 0.
    \end{cases}
\end{equation*}
This is equivalent to
\begin{equation}
    \label{eq:result2}
    \begin{cases}
        f_{4}(v_{a_1, s_1}) = f_{4}(v_{a_2, s_2})  = 0, \\
        f_{6}(v_{a_1, s_1})  f_{6}(v_{a_2, s_2}) < 0.
    \end{cases}
\end{equation}
Finally, we consider Case (c).
Applying the same arguments again, the augmented coefficient matrix can be reduced to
\[
    \begin{bmatrix}
        1 & 1 & 1                     \\
        0 & 0 & - f_{4}(v_{a_1, s_1}) \\
        0 & 0 & - f_{6}(v_{a_1, s_1}) \\
    \end{bmatrix}.
\]
Thus, in Case (c), a weighted 7-design of type (\ref{eq:invariant2orbit}) exists if and only if
\begin{equation}
    \label{eq:result3}
        f_{4}(v_{a_1, s_1}) = f_{4}(v_{a_2, s_2}) = f_{6}(v_{a_1, s_1}) =  f_{6}(v_{a_2, s_2}) = 0.
\end{equation}

By summarizing (\ref{eq:result1})-(\ref{eq:result3}), we obtain all the desired conditions.
\end{proof}

\begin{theorem}
   \label{thm:main1}
   Let $n \geq 3$.
   Case (iii) of Proposition~\ref{prop:Hirao2024-1} does not occur if $s_1 \ne s_2$.
\end{theorem}

We prove Theorem \ref{thm:main1} in Section~\ref{sect:proof}.

\begin{remark}
We list $\tilde{W}_1, \tilde{W}_2$ of Cases (i), (ii) listed in Table \ref{tbl:7designW}
\begin{table}[htb]
\begin{center}
\caption{Weights of $7$-design}
\label{tbl:7designW}
\begin{tabular}{ccc}\hline
 &$\tilde{W}_1$&$\tilde{W}_2$\\\hline
 Case (i) 
 & $\frac{f_{4}(v_{a_2, s_2})}{f_{4}(v_{a_2, s_2}) - f_{4}(v_{a_1, s_1}) }$ 
 & $- \frac{f_{4}(v_{a_1, s_1})}{f_{4}(v_{a_2, s_2}) - f_{4}(v_{a_1, s_1}) }$\\ \hline
 Case (ii) 
 & $\frac{f_{6}(v_{a_2, s_2})}{f_{6}(v_{a_2, s_2}) - f_{6}(v_{a_1, s_1})}$ 
 & $- \frac{f_{6}(v_{a_1, s_1})}{f_{6}(v_{a_2, s_2}) - f_{6}(v_{a_1, s_1})}$ \\ \hline
\end{tabular}
\end{center}
\end{table}
\end{remark}

\subsection{Characterization of $9$-designs}\label{subsect:2orb_deg9}

We start with 
a characterization theorem for $9$-designs with two orbits. 

   \begin{proposition}
      \label{prop:Hirao2024-2}
     Let $n \geq 3$.
     A weighted $9$-design of type (\ref{eq:invariant2orbit}) exists if and only if one of the following cases (i)-(v) holds:
      \begin{enumerate}
         \item[(i)]
               \begin{align*}
                  f_{4}(v_{a_1, s_1}) f_{4}(v_{a_2, s_2}) < 0
                 ,\; G_{4, 6}(a_1, a_2, s_1, s_2) = G_{4, 8, 1}(a_1, a_2, s_1, s_2) = G_{4, 8, 2} (a_1, a_2, s_1, s_2) = 0.
               \end{align*}
         \item[(ii)]
               \begin{align*}
                   & f_{4}(v_{a_1, s_1}) =  f_{4}(v_{a_2, s_2}) = 0,\ f_{6}(v_{a_1, s_1}) f_{6}(v_{a_2, s_2}) < 0, \\
                   & \quad\;
                  G_{6, 8, 1}(a_1, a_2, s_1, s_2) = G_{6, 8, 2} (a_1, a_2, s_1, s_2)= 0.
               \end{align*}
         \item[(iii)]
               \begin{align*}
                   & f_{4}(v_{a_1, s_1}) =  f_{4}(v_{a_2, s_2}) = f_{6}(v_{a_1, s_1}) =  f_{6}(v_{a_2, s_2}) = 0, \\
                   & \quad\;  f_{8,1}(v_{a_1, s_1}) f_{8,1}(v_{a_2, s_2}) < 0 \;,\;
                  G_{8, 1, 8, 2}(a_1, a_2, s_1, s_2) = 0.
               \end{align*}
         \item[(iv)]
               \begin{align*}
                   & f_{4}(v_{a_1, s_1}) =  f_{4}(v_{a_2, s_2}) = f_{6}(v_{a_1, s_1}) =  f_{6}(v_{a_2, s_2}) =
                  f_{8,1}(v_{a_1, s_1}) = f_{8,1}(v_{a_2, s_2}) =  0,                                           \\
                   & \quad\; f_{8,2}(v_{a_1, s_1}) f_{8,2}(v_{a_2, s_2}) < 0.
               \end{align*}
         \item[(v)]
               \begin{align*}
                   & f_{4}(v_{a_1, s_1}) =  f_{4}(v_{a_2, s_2}) = f_{6}(v_{a_1, s_1}) =  f_{6}(v_{a_2, s_2}) = 0,        \\
                   & f_{8,1}(v_{a_1, s_1}) = f_{8,1}(v_{a_2, s_2}) =  f_{8,2}(v_{a_1, s_1}) = f_{8,2}(v_{a_2, s_2}) =  0.
               \end{align*}
      \end{enumerate}
\end{proposition}
 
\begin{remark}
    When $n=3$, we may ignore all conditions involving $f_{8,2}$.
    For example, both Case (iv) and Case (v) can be reduced to
    \begin{equation*}
        f_{4}(v_{a_1, s_1}) =  f_{4}(v_{a_2, s_2}) = f_{6}(v_{a_1, s_1}) =  f_{6}(v_{a_2, s_2}) = 0,        \\
        \; \; f_{8,1}(v_{a_1, s_1}) = f_{8,1}(v_{a_2, s_2}) =  0.
    \end{equation*}
\end{remark}

\begin{theorem}\label{thm:main3}
   Cases (iii), (iv) and (v) of Proposition~\ref{prop:Hirao2024-2} do not occur.
\end{theorem}

We prove  Theorem~\ref{thm:main3} in Section~\ref{sect:proof}, 
where a proof of Proposition~\ref{prop:Hirao2024-2} is given in the Appendix. 

\bigskip

We close this subsection by discussing weighted $9$-designs of type $(v_{a_1, s_1}^{B_3} \cup v_{a_2, s_2}^{B_3}, \{ \tilde{W}_1, \tilde{W}_2 \})$.
Such $9$-designs can be classified by two examples listed in Table~\ref{tbl:weighted 9-design}, where numerical parameters are
considered with six significant digits.
(Throughout this paper, numerical results are treated with the same level of precision.)
What is remarkable here is that all these examples belong to Case (ii), 
which implies that
Case (i) may not occur even for $n \ge 4$.
\begin{table}[htb]
\begin{center}
\caption{Parameters of weighted 9-design}
\label{tbl:weighted 9-design}
\begin{tabular}{ccccccc} \hline
$s_1$ & $s_2$ & $a_1$ & $a_2$ & $W_1$ & $W_2$  \\ \hline 
1 & 2 & 0.396751 & 0.470350 & 0.0220088 & 0.0196579 \\ \hline
2 & 2 & 0.389041 & 3.89103 & 0.0193973 & 0.0222694 \\ \hline
\end{tabular}
\end{center}
\end{table}

The orbit $v_{a,1}^{B_3}$ can be understood as the vertices of a convex polyhedron
obtained by uniformly removing
the $6$ square pyramid parts from a regular octahedron, 
or by uniformly \textit{truncating} the $6$ vertices of a regular octahedron in such a way that each edge of the octahedron is divided into the ratio $a:1-a:a$ if $0 < a < 1$, and $1:a-1:1$ if $1 <a$. 
The values of $a_1$, $a_2$, and $W_1$,
appearing in Case 1 (see Figure~\ref{fig:9-design1}),
satisfy the equations
\[
\begin{cases}
17 - 92 A_2 + 78 A_2^2 - 44 A_2^3 + 5 A_2^4 = 0, \\
24 - 167 A_1 + 24 A_1^2 + 57 A_1 A_2 - 39 A_1 A_2^2 + 5 A_1 A_2^3 = 0,\\
-845 - 522 A_2 + 369 A_2^2 - 40 A_2^3 + 1785 W_1 = 0,
\end{cases}
\]
where $A_1 = a_1^2$ and $A_2 = a_2^2$.

\begin{figure}[htb]
\includegraphics[width=14cm]{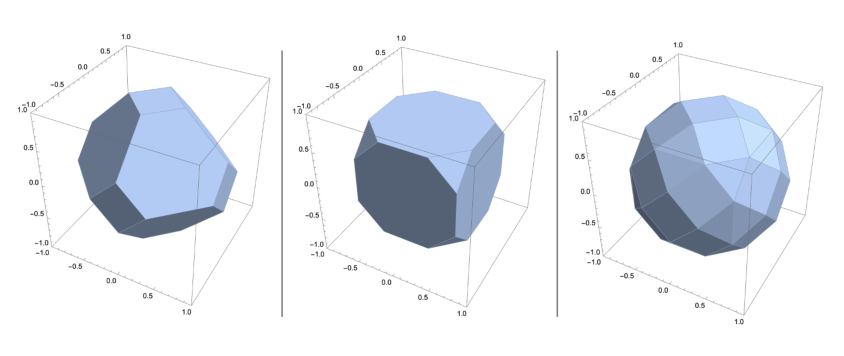}
\caption{$v_{a_1, 1}^{B_3}$(left), $v_{a_2, 2}^{B_3}$(middle), $v_{a_1, 1}^{B_3} \cup v_{a_2, 2}^{B_3}$(right)}\label{fig:9-design1}
\end{figure}

The orbit $v_{a,2}^{B_3}$ can be characterized in terms of vertex-truncation and edge-truncation.
When $0 < a < 1$, the orbit $v_{a,2}^{B_3}$ can be obtained by uniformly removing the $8$ rectangular triangular pyramids from a cube in such a way that each edge of the cube is divided into the ratio $1-a:2a:1-a$. If $1 < a$, the orbit $v_{a,2}^{B_3}$ can be obtained by uniformly cutting off the $8$ right-angle isosceles prisms including the edges of a cube in such a way that each edge of the cube is divided into the ratio $a-1:2:a-1$. 
The values of $a_1$, $a_2$, and $W_1$,
appearing in Case 2 (see Figure~\ref{fig:9-design2}), 
satisfy the equations
\[
\begin{cases}
-19 + 116 A_2 + 66 A_2^2 - 20 A_2^3 + A_2^4 = 0, \\
-175 + 12 A_1 - 47 A_2 + 19 A_2^2 - A_2^3 = 0, \\
-3956 + 2291 A_2 - 349 A_2^2 + 13 A_2^3 + 7770 W_1 = 0,
\end{cases}
\]
where $A_1 = a_1^2$ and $A_2 = a_2^2$.

\begin{figure}[H]
\label{fig2}
\includegraphics[width=14cm]{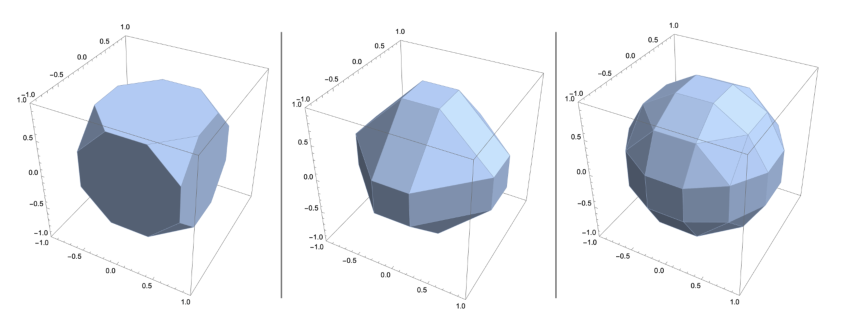}
\caption{$v_{a_1, 2}^{B_3}$(left), $v_{a_2, 2}^{B_3}$(middle), $v_{a_1, 2}^{B_3} \cup v_{a_2, 2}^{B_3}$(right)}\label{fig:9-design2}
\end{figure}

Note that all $24$-point configurations consisting of each single proper orbit are 3-designs.

\subsection{Proofs of Theorems~\ref{thm:main1} and~\ref{thm:main3}} 
\label{sect:proof}

\begin{proof}[Proof of Theorem~\ref{thm:main1}]
   Suppose that $s_1 \ne s_2$. 
   By the similar argument that used in Proposition~\ref{prop:Tanino1}, 
   it holds that
    \begin{equation*}
       \begin{gathered}
          h(n, s_i) = n^3 + (2-9s_i)n^2 + (12s_i^2 -9s_i -7)n + 6s_i^2 + 18s_i + 4 \\
          \qquad \qquad \pm (n^3 - 3(s_i-2) -3s_i-7) \sqrt{(n+2)s_i(1-n+3s_i)} = 0, \quad i=1,2.
       \end{gathered}
    \end{equation*}
Since 
$(n^3 + (2-9s_i)n^2 + (12s_i^2 -9s_i -7)n + 6s_i^2 + 18s_i + 4)^2 = 
((n^3 - 3(s_i-2) -3s_i-7) \sqrt{(n+2)s_i(1-n+3s_i)})^2$ for each $i$,
by noting $n \geq 3$ and $s_i \geq 0$, we obtain 
   \begin{equation*}
      \begin{gathered}
         0 = n^4 + (6-9s_i)n^3 + (1-30s_i+27s_i^2)n^2 - (24-9s_i-54s_i^2+27s_i^3)n \\
         \qquad - 18s_i^3 - 36s_i^2 + 30s_i + 16, \quad i=1,2.
      \end{gathered}
   \end{equation*}
   By subtracting them and then dividing it by $3(s_2-s_1)$, we obtain the Diophantine equation
   \begin{equation}
      \label{eq:Tanino4}
      \begin{gathered}
         0 = 3n^3 + (10-9s_1-9s_2)n^2 -3 (1+6s_1+6s_2-3s_1^2-3s_1s_2-3s_2^2)n \\
         \qquad -10+12s_1+12s_2+6s_1^2+6s_1s_2+6s_2^2.
      \end{gathered}
   \end{equation}
   By using Mathematica, the positive integer solutions of (\ref{eq:Tanino4}) are $(n, s_1, s_2) = (8, 3, 3)$, $(5, 2, 2)$, $(2, 1, 1)$, each of which does not satisfy the restriction $s_1 \ne s_2$.
\end{proof}

Next we prove Theorem~\ref{thm:main3} by focusing on Cases  (iii), (iv), and (v) of Proposition~\ref{prop:Hirao2024-2}. 

We first consider the case where $s_1 \neq s_2$.
In this case, Theorem~\ref{thm:main1} implies the nonexistence of a weighted $9$-design.
Thus, our interest goes to the case where $s_1 = s_2$. 

To prove Theorem~\ref{thm:main3}, 
we prepare the following two lemmas.
\begin{lemma}\label{lem:range-s}
Let $n \geq 3$, $a_1\neq a_2$ and 
$s_1 = s_2$ (say $s$).
Suppose $f_4(v_{a_1,s}) = f_4(v_{a_2, s}) = f_{6}(v_{a_1, s}) = f_{6}(v_{a_2,s})= 0$.
Then
\begin{equation*}
\frac{5n^2 + 15 n - 20}{12 n + 6} < s \leq \frac{5n^2 + 15 n - 20}{9n + 12}.
\end{equation*}
\end{lemma}
\begin{proof}
Suppose (\ref{eq:zeros1}) with $a_1 \neq a_2$ and 
$s = s_1 = s_2$.
Since $\tilde{f}_4(v_{a_1, s}) - \tilde{f}_4(v_{a_2,s}) = 0$, we have
   \begin{equation}
      \label{eq:s1s2-1}
      a_1^2 + a_2^2 =  \frac{6s}{n-1}.
   \end{equation}
Since $\tilde{f}_6(v_{a_1, s}) - \tilde{f}_6(v_{a_2,s}) = 0$, we have
\begin{equation}
\label{eq:s1s2-2}
(a_1^4 + a_1^2 a_2^2 + a_2^4) (-2 + n) (-1 + n) - 60 s - 
 15 (a_1^2 + a_2^2) (-2 + n) s - 15 n s + 90 s^2 = 0.
\end{equation}

Substituting (\ref{eq:s1s2-1}) into (\ref{eq:s1s2-2}), 
we have
\begin{equation}
\label{eq:s1s2-3}
a_1^2 a_2^2 = \frac{3 s \{ 20 - 5 n^2 + 6 s + 3 n (-5 + 4 s)\}}{(-2 + n) (-1 + n)^2}.
\end{equation}
Since $a_1^2 a_2^2 >0$, we obtain 
\begin{equation}
\label{eq:s1s2-4}
\frac{5n^2 + 15 n - 20}{12 n + 6} < s \leq n-1.
\end{equation}
Meanwhile, solving for $a_1^2$ and $a_2^2$ from (\ref{eq:s1s2-1}) and (\ref{eq:s1s2-3}), we know that 
$a_1^2$ (or $a_2^2$) is equal to  
\begin{equation*}
\frac{3s}{n-1} \pm \frac{\sqrt{3 (-2 + n) s \{ 5 (-1 + n) (4 + n) - 3 (4 + 3 n) s \}}}{(-2 + n) (-1 + n)}.
\end{equation*}
Since $a_1^2$ and $a_2^2$ are positive reals, 
by noting $0 \leq s \leq n-1$, we have
\begin{equation}\notag
s \big (5 n^2 + n (15 - 9 s) - 4 (5 + 3 s) \big ) \geq 0 \; \quad \text{and} \quad \; 0 \leq s \leq n -1 .
\end{equation}
Then we obtain
\begin{equation}
\label{eq:s1s2-5}
    0 \leq s \leq \frac{5 n ^2 + 15 n -20}{9n+12}.
\end{equation}
Thus, by combining (\ref{eq:s1s2-4}) and (\ref{eq:s1s2-5}), 
we obtain the desired result.
\end{proof}

\begin{lemma}\label{lem:6s-3}
   Assume $a_1 \neq a_2$. Then the following hold:

   \begin{enumerate}
       \item[(i)] 
   If $n \geq 4$ and
   $f_4(v_{a_1,s}) = f_4(v_{a_2, s}) = f_{8,1}(v_{a_1, s}) = f_{8,1}(v_{a_2,s}) = f_{8,2}(v_{a_1, s}) = f_{8,2}(v_{a_2,s}) = 0$,
   then 
   \[
      n = 6 s - 3.
   \]
   \item[(ii)]
   If $n = 3$ and
   $f_4(v_{a_1,s}) = f_4(v_{a_2, s}) = f_{8,1}(v_{a_1, s}) = f_{8,1}(v_{a_2,s}) = 0$,
   then
   \begin{equation}\label{eq:a1a2}
      a_1=\sqrt{\frac{3\mp\sqrt{5}}{2}},\ a_2=\frac{\pm 1+\sqrt{5}}{2},\ s=1.
   \end{equation}
 
   \end{enumerate}

\end{lemma}
\begin{proof}
Suppose $n\geq 4$.\ 
   By recalling (\ref{eq:s1s2-1}),  $\tilde{f}_4(v_{a_1, s}) - \tilde{f}_4(v_{a_2,s}) = 0$ and $\tilde{f}_{8, i} (v_{a_1, s}) - \tilde{f}_{8, i} (v_{a_2, s}) = 0 \; (i = 1, 2)$,
   we have
   \begin{align*}
      g_1(a_1,a_2,s) & = - (a_1^2 + a_2^2) + (a_1^2 + a_2^2) n - 6 s = 0,                     \\
      g_2(a_1,a_2,s) & = -a_1^6 - a_1^4 a_2^2 - a_1^2 a_2^4 - a_2^6 + a_1^6 n + a_1^4 a_2^2 n \\
      +              & a_1^2 a_2^4 n + a_2^6 n - 28 s + 70 a_1^2 s - 28 a_1^4 s + 70 a_2^2 s -
      28 a_1^2 a_2^2 s - 28 a_2^4 s = 0,                                                       \\
      g_3(a_1,a_2,s) & = -18 + 3 a_1^2 + 3 a_2^2 - 6 n + 2 a_1^2 n + 2 a_2^2 n - a_1^2 n^2    \\
      -              & a_2^2 n^2 + 36 s - 9 a_1^2 s - 9 a_2^2 s + 6 n s + 3 a_1^2 n s +
      3 a_2^2 n s - 18 s^2 = 0.
   \end{align*}
   By computing a Groebner basis $\mathcal{G}$ for the ideal $\langle g_1,g_2,g_3 \rangle$,
   we can check that
   \begin{align}
      \{ & (3 + n - 6 s) (-1 + n - s),\                                              
         (5 - 4 a_2 + a_2^2) (5 + 4 a_2 + a_2^2) (-1 + n) (3 + n - 6 s) \} \subset \mathcal{G}
      \label{eq:gbasis1}
   \end{align}

   Thus it holds that
   \begin{equation*}
       n \in \{6s-3, s+1\}.
   \end{equation*}
By substituting $n = s + 1$ into  (\ref{eq:gbasis1}),
   we have
   \[
      (5 - 4 a_2 + a_2^2) (5 + 4 a_2 + a_2^2) (4 - 5 s) s = 0.
   \]
   This equation holds true only when $s=0$,
   but this is a contradiction.
   Thus we obtain $n = 6s - 3$.

Next,\ suppose $n=3$.\  By computing a Groebner basis $\mathcal{H}$ for the ideal $\langle g_1,g_2 \rangle$,
   we have
   \begin{align*}
      h_1(a_1,a_2,s) =s(14+8a_2^4-105s-24a_2^2s+99s^2)=0,\;
      h_2(a_1,a_2,s) =a_1^2+a_2^2-3s=0.
   \end{align*}
    We solve $14+8a_2^4-105s-24a_2^2s+99s^2=0$.
    This equation holds only when 
    \begin{equation*}
        a_2^2=\frac{3s}{2}\pm \frac{\sqrt{-14+105s-81s^2}}{2\sqrt{2}},\quad  
   \frac{35-\sqrt{721}}{54}<s<\frac{35+\sqrt{721}}{54}.
    \end{equation*}
    
   Therefore, we obtain $s=1$ and
   \begin{equation*}
      a_1=\sqrt{\frac{3\mp\sqrt{5}}{2}},\ a_2=\frac{\pm 1+\sqrt{5}}{2}.
   \end{equation*}

\end{proof}

We are ready to prove Theorem~\ref{thm:main3}.
\begin{proof}[Proof of Theorem~\ref{thm:main3}]
Let $n \geq 3$, and
suppose that there exists a weighted 9-design of type~(\ref{eq:invariant2orbit}).
By Theorem~\ref{thm:main1} we may consider the case where $s_1 = s_2$ (say $s$).

First, we consider Case (iii).
Then 
\begin{equation}
\begin{cases}
    f_{4}(v_{a_1, s}) =  f_{4}(v_{a_2, s})
    = f_{6}(v_{a_1, s}) =  f_{6}(v_{a_2, s}) = 0, \\ f_{8,1}(v_{a_1, s}) f_{8,1}(v_{a_2, s}) < 0.
\end{cases}
    \label{eq:f81}
\end{equation}
Then by Lemma~\ref{lem:range-s}, we have
    \begin{equation}
\label{eq:range-s}
\frac{5n^2 + 15 n - 20}{12 n + 6} < s \leq \frac{5n^2 + 15 n - 20}{9n + 12}.
\end{equation}
By using Mathematica, we find that $n$ satisfying  (\ref{eq:f81}) and (\ref{eq:range-s}) is upper bounded by 691.
However, for $n \leq  691$, we can also confirm that there exists no integer $s$ such that 
\[
f_{4}(v_{a_1, s}) = f_{6}(v_{a_2, s}) = G_{8,1,8,2}(a_1, a_2, s, s) = 0
\]

Next, we consider Cases (iv) and (v).
   Let $n\geq 4$.
   By Theorem~\ref{thm:Bajnok} we may assume that $s>0$.
   By combining (\ref{eq:s1s2-1}), (\ref{eq:s1s2-2}) and Lemma~\ref{lem:6s-3},
   we obtain
   \[
      a_1^2 a_2^2 = - \frac{3 s (-5 - 15 s + 27 s^2)}{(-2 + 3 s)^2 (-5 + 6 s)} < 0,
   \]
   which is a contradiction.
   Let $n=3$.
   Substituting $a_1, a_2, s$ as in
   (\ref{eq:a1a2})
   into  (\ref{eq:s1s2-2}), we have
   \begin{equation*}
(a_1^4 + a_1^2 a_2^2 + a_2^4) (-2 + n) (-1 + n) - 60 s - 
 15 (a_1^2 + a_2^2) (-2 + n) s - 15 n s + 90 s^2 = -44
 \neq 0,
\end{equation*}
which is again a contradiction.
\end{proof}

\section{Designs with more than two proper orbits on $\mathbb{S}^3$}\label{sect:3orbit}

In Theorem~\ref{thm:main_bound}, we have obtained a uniform bound for the existence of designs of type (\ref{eq:invariant0}), namely we have shown that if $n \ge 4$, and if there exists a $t$-design on $\mathbb{S}^{n-1}$ of type (\ref{eq:invariant0}), then $t \le 15$.

Schur's formula (\ref{eq:Schur}) on $\mathbb{S}^3$ has degree $11$.
Sawa and Xu~\cite{SX2013} discusses
a higher-dimensional extension of Schur's design, and discovers
many examples of weighted $11$-designs
in dimensions $3$ through $23$.
Note that all these designs contain only one proper orbit.
Sawa and Xu~\cite{SX2013} moreover establishes
that designs of type (\ref{eq:invariant0}) with only one proper orbit have degree at most $11$.

Then what about designs with at least two proper orbits?
The following is a weighted $11$-design on $\mathbb{S}^{3}$ with $3$ proper orbits:

\begin{itemize}
\item $(v_{a_1,2}^{B_4} \cup v_{a_2, 3}^{B_4} \cup v_{a_3, 3}^{B_4}, \{ W_1, W_2, W_3 \} )$, 
\[
\begin{cases}
a_1 = 0.470499, \quad a_2 = 1.17310, \quad a_3 = 3.78381, \\
W_1 = 0.00522948, \quad W_2 = 0.00192016, \quad W_3 = 0.00586062.
\end{cases}
\]
\end{itemize}
We could not have found $11$-designs, and even $9$-designs with $2$ proper orbits.

The use of more than $3$ proper orbits enables us to obtain $13$-designs. Indeed, the following two are weighted $13$-designs
on $\mathbb{S}^3$ with $5$ orbits:
\begin{itemize}
\item 
$(v_{1,4}^{B_4} \cup v_{a_2, 1}^{B_4} \cup v_{a_3, 2}^{B_4} \cup v_{a_4, 3}^{B_4} \cup v_{a_5, 3}^{B_4}, \{ W_1, W_2, W_3, W_4, W_5 \})$,
\[
\begin{cases}
a_2 = 0.444883, \quad a_3 = 0.509692, \quad a_4 = 9.68607, \quad a_5 = 2.53788, \\
W_1 = 0.00475002, \quad W_2 = 0.00407904, \quad W_3 = 0.00435065, \\ 
W_4 = 0.000536920, \quad W_5 = 0.00431532.
\end{cases}
\]

\item $(v_{1,0}^{B_4} \cup v_{a_2, 1}^{B_4} \cup v_{a_3, 2}^{B_4} \cup v_{a_4, 2}^{B_4} \cup v_{a_5, 3}^{B_4}, \{ W_1, W_2, W_3, W_4, W_5 \})$,
\[
\begin{cases}
a_2 = 0.597599, \quad a_3 = 0.521840, \quad a_4 = 3.07756, \quad a_5 = 1.85216, \\
W_1 = 0.00262253, \quad W_2 = 0.00244207, \quad W_3 = 0.00370960, \\
W_4 = 0.00256322, \quad W_5 = 0.00405640.
\end{cases}
\]
\end{itemize}

At this point we could not have succeeded in finding a $15$-design on $\mathbb{S}^3$ with more than $5$ proper orbits, though there actually exist $15$-designs with negative weights $W_i$ (see for example Keast~\cite{Keast1987}). 
The following is a challenging open question, which is left for future work.

\begin{problem}\label{prob:15design}
Does there exist a weighted $15$-design
on $\mathbb{S}^3$ of type (\ref{eq:invariant0}) with more than $5$ proper orbits?
\end{problem}

The situation is quite different in dimensions $3$ and $4$. Indeed, we obtain weighted $9$-designs with two proper orbits, as already seen in
Section~\ref{subsect:2orb_deg9} (Table~\ref{tbl:weighted 9-design}). 
As in the four-dimensional case, one can obtain weighted $11$-designs on $\mathbb{S}^{2}$ with $3$ proper orbits, e.g.
\begin{itemize}
\item $(v_{a_1, 1}^{B_3} \cup v_{a_2, 1}^{B_3} \cup v_{a_3, 2}^{B_3}, \{ W_1,  W_2, W_3 \})$, 
\[
\begin{cases}
a_1 = 1.51610, \quad a_2 = 4.44671, \quad a_3 = 1.79726, \\
W_1 = 0.0147251, \quad W_2 = 0.00936568, \quad W_3 = 0.0175759.
\end{cases}
\]
\end{itemize}
A remarkable gap between the three and four dimensional cases is that there exists a weighted $17$-design on $\mathbb{S}^2$ of type (\ref{eq:invariant0}). Indeed, the following example with $6$ proper orbits can be found in Table 2.1 of Heo-Xu~\cite[p.275]{HX2001}:
\begin{itemize}
\item 
$(v_{1,0}^{B_3} \cup v_{1, 2}^{B_3} \cup v_{a_3, 1}^{B_3} \cup v_{a_4, 1}^{B_3} \cup v_{a_5, 2}^{B_3} \cup v_{a_6, 2}^{B_3}, \{ W_1, W_2, W_3, W_4, W_5, W_6 \})$, 
\[
\begin{cases}
a_3 = 0.645826, \quad a_4 = 0.228616, \quad a_5 = 2.41421, \quad a_6 = 0.414214, \\
W_1 = 0.00266400, \quad W_2 = 0.00910340, \quad W_3 = 0.00979854, \\
W_4 = 0.00955987, \quad W_5 = 0.0107778, \quad W_6 = 0.00916195.
\end{cases}
\]

\end{itemize}

\section{Conclusion and further remarks}\label{sect:conclusion}

In this paper we have explored a full generalization of the classical corner-vector method for constructing weighted spherical designs, and 
have extensively studied
the existence of designs of type (\ref{eq:invariant0}). We have first established a uniform upper bound for the degree of such designs (Theorem~\ref{thm:main_bound}). Our proof is a hybrid argument combining the cross-ratio comparison technique for Hilbert identity and Xu's theorem on the interrelation between spherical designs and simplicial designs, which appears to be useful in the study of designs of a different type than (\ref{eq:invariant0}).
Moreover, we have made a detailed observation about the existence of $7$-designs with one proper orbit, $9$-designs with two proper orbits, and $11$- and $13$-designs with more than two proper orbits.

In the rest of this paper we explore the connections between integral lattices and some of our designs.
Since the pioneering paper by Venkov~ \cite{V1984}, there have been numerous publications on the construction that derives spherical designs from  shells of
integral lattices. Such {\it lattice-based constructions} are also significant for coding theorists as a spherical analogue of the Assmus-Mattson theorem that directly relates blocks designs to linear codes.
For a brief introduction to the connections among designs, codes and lattices, we refer the reader to Bannai and Bannai~\cite{BB2009}.

We shall make a brief explanation of the lattice-based construction, together with examples for $D_4$-root lattice.
The {\it $D_4$-root lattice}, its dual lattice $D_4^*$ and $\sqrt{2}$-normalization $D_4'$ are defined by
\[
\begin{gathered}
D_4 = \{ (x_1,x_2,x_3,x_4) \in \mathbb{Z}^4 \mid x_1 + x_2 + x_3 + x_4 \equiv 0 \pmod{2} \}, \\
D_4^* = \{ (y_1,y_2,y_3,y_4) \in \mathbb{R}^4 \mid \sum_{i=1}^4 x_iy_i  \in \mathbb{Z}, \; (x_1, x_2, x_3, x_4) \in D_4 \}, \\
D_4' = \sqrt{2} D_4^*.
\end{gathered}
\]
Given a set $E$ of $\mathbb{R}^4$, the {\it shell of (square) norm $m$}, say $E_m$, 
is defined to be the intersection of $E$ and the concentric sphere $\mathbb{S}_m^3$ of radius $m$.
We also use this terminology for the $n$-dimensional case.

\begin{example}\label{ex:D4shell}
Let 
$E = D_4 \cup D_4'$.
Then the first two shells for $D_4$ are given by
\[
(D_4)_2 = (1,1,0,0)^{B_4}, \quad
(D_4)_6 = (2,1,1,0)^{B_4},
\]
which can be represented in terms of generalized corner vectors as $(\sqrt{2}v_{1,1})^{B_4}$ and $(\sqrt{6}v_{2,2})^{B_4}$, respectively.
Similarly, the first two shells of $D_4'$ are given as
\[
\begin{gathered}
(D_4')_2 = (\sqrt{2},0,0,0)^{B_4} \cup (\frac{1}{\sqrt{2}},\frac{1}{\sqrt{2}},\frac{1}{\sqrt{2}},\frac{1}{\sqrt{2}})^{B_4} = (\sqrt{2}v_{1,0})^{B_4} \cup (\sqrt{2}v_{1,3})^{B_4}, \\
(D_4')_6 = (\sqrt{2},\sqrt{2},\sqrt{2},0)^{B_4} \cup (\frac{3}{\sqrt{2}},\frac{1}{\sqrt{2}},\frac{1}{\sqrt{2}},\frac{1}{\sqrt{2}})^{B_4} = (\sqrt{6}v_{1,2})^{B_4} \cup (\sqrt{6}v_{3,3})^{B_4}.
\end{gathered}
\]
\end{example}

The following theorem is proved in de la Harpe et al.~\cite{HPV2007}.

\begin{theorem}[Theorem $D_4$, \cite{HPV2007}]\label{thm:HPV2007}
With the notation given above, the following holds:
\begin{enumerate}
\item[(i)] Any shell $E_{2m}$ is an (equi-weighted) spherical $7$-design.
\item[(ii)] The union of $\frac{1}{\sqrt{2}}E_2$ and $\frac{1}{\sqrt{6}}E_6$ is a weighted spherical $11$-design.
\end{enumerate}
\end{theorem}

\begin{remark}\label{rem:D4shell}
Theorem~\ref{thm:HPV2007}~(i) is a corollary of Bajnok's theorem (Theorem~\ref{thm:Bajnok}), and Theorem~\ref{thm:HPV2007}~(ii) is just Schur's formula described in Example~\ref{ex:spherical_dim4}.
\end{remark}

We now clarify the connection between our $7$-design $v_{2,8}^{B_{16}}$ (see Remark~\ref{thm:Tanino}) and a certain integral lattice called the {\it Barnes-Wall lattice} in $\mathbb{R}^{16}$.
The Barnes-Wall lattice in $\mathbb{R}^2$ is the standard $\mathbb{Z}^2$-lattice, and the Barnes-Wall lattice in $\mathbb{R}^4$ is just the $D_4$-root lattice. 
In general the Barnes-Wall lattice in $\mathbb{R}^{2^k}$ can be realized as the rational part of the lattice $M_1^{\otimes_k}$, where $M_1$ is the set of all $\mathbb{Z}[\sqrt{2}]$-integer combinations of $(\sqrt{2},0)$ and $(1,1)$
(see~\cite[Theorem 2.1]{NRS2002} for the details).
In particular the lattice in $\mathbb{R}^{16}$ is generated by the rows of the matrix
\begin{equation*}
X =
\left[\begin{array}{cccccccccccccccc}
4&0&0&0&0&0&0&0&0&0&0&0&0&0&0&0 \\
4&4&0&0&0&0&0&0&0&0&0&0&0&0&0&0 \\
4&0&4&0&0&0&0&0&0&0&0&0&0&0&0&0 \\
2&2&2&2&0&0&0&0&0&0&0&0&0&0&0&0 \\
4&0&0&0&4&0&0&0&0&0&0&0&0&0&0&0 \\
2&2&0&0&2&2&0&0&0&0&0&0&0&0&0&0 \\
2&0&2&0&2&0&2&0&0&0&0&0&0&0&0&0 \\
2&2&2&2&2&2&2&2&0&0&0&0&0&0&0&0 \\
4&0&0&0&0&0&0&0&4&0&0&0&0&0&0&0 \\
2&2&0&0&0&0&0&0&2&2&0&0&0&0&0&0 \\
2&0&2&0&0&0&0&0&2&0&2&0&0&0&0&0 \\
2&2&2&2&0&0&0&0&2&2&2&2&0&0&0&0 \\
2&0&0&0&2&0&0&0&2&0&0&0&2&0&0&0 \\
2&2&0&0&2&2&0&0&2&2&0&0&2&2&0&0 \\
2&0&2&0&2&0&2&0&2&0&2&0&2&0&2&0 \\
1&1&1&1&1&1&1&1&1&1&1&1&1&1&1&1 \\
\end{array} \right].
\end{equation*}

\begin{theorem}\label{thm:BW16}
The $7$-design $(8 \sqrt{3} v_{2,8})^{B_{16}}$
is a subset of some shell of the Barnes-Wall lattice in $\mathbb{R}^{16}$.
\end{theorem}
\begin{proof}
We note that $4e_1,\ldots,4e_{16}$, where $e_1,\ldots,e_{16}$ are the standard basis vectors, are all integer combinations of the rows of the generator matrix $X$.
The orbit $(8 \sqrt{3} v_{2,8})^{B_{16}} = (8, \underbrace{4, \ldots, 4}_{8 \text{ times }}, 0, \ldots, 0)^{B_{16}}$
is thus included in some shell of the lattice.
\end{proof}

Next we come to the connection between our $7$-design $v_{2,11}^{B_{23}}$ (see Remark~\ref{thm:Tanino}) and a certain integral lattice in $\mathbb{R}^{23}$.
The {\it Leech lattice} $\Lambda_{24}$ is an even unimodular lattice, which often appears to be the densest sphere packing in $\mathbb{R}^{24}$.
As briefly explained in~\cite[Chapter~24]{CS1999}, 
the lattice $\Lambda_{24}$ has an explicit expression as
\begin{equation*}
\begin{gathered}
\Lambda_{24} =
\Big\{ \frac{1}{\sqrt{8}} (2\boldsymbol{c} + 4\boldsymbol{x}) \mid \boldsymbol{c} \in G_{24} \pmod{2},\; \boldsymbol{x} \in 
\mathbb{Z}^{24}
\ \text{with} \ \sum_{i=1}^{24} x_i \equiv 0 \pmod{2} \Big\} \\
\cup
\Big\{ \frac{1}{\sqrt{8}} (\boldsymbol{1} + 2\boldsymbol{c} + 4\boldsymbol{x}) \mid \boldsymbol{c} \in G_{24} \pmod{2},\; \boldsymbol{x} \in 
\mathbb{Z}^{24}
\ \text{with} \ \sum_{i=1}^{24} x_i \equiv 0 \pmod{2} \Big\}
\end{gathered}
\end{equation*}
where $\boldsymbol{1}$ is the $24$-dimensional all-one vector and
$G_{24}$ is the set of all $24$-dimensional vectors, considered as vectors in $\mathbb{R}^{24}$, of the \textit{extended Golay codes} over $\mathbb{F}_2$; see, for example, van Lint and Wilson~\cite[\S~20]{LW2001} for the definition of the (extended) Golay code over $\mathbb{F}_2$.
The minimal shell consists of $1104$ points of type $(\pm 4^2, 0^{22})$, $97152$ points of type $(\pm 2^8, 0^{16})$, $98304$ points of type $(\mp 3^1, \pm 1^{23})$, and totally $196560$ points.

The {\it shorter Leech lattice}, $O_{23}$, is the unique
odd (up to isometry)
unimodular lattice 
with minimal norm $3$
in $\mathbb{R}^{23}$~\cite[Chapter~19]{CS1999}. 
Given a minimal vector $ v \in \Lambda_{24}$ of norm $4$, 
the lattice $O_{23}$ can also be identified as the orthogonal projection of the set of points in $\Lambda_{24}$ that have an even inner product with $v$ onto
\[
v^{\perp} = \{ u \in \mathbb{R}^{24} \mid \sum_{i=1}^{24} u_iv_i = 0 \};
\]
see for example~\cite[p.179]{CS1999}.

Consider a hyperplane $H$ given by
\begin{equation*}
    H := \{ 
    (x_1,\ldots,x_{23},0) \in \mathbb{R}^{24} \mid x_1,\ldots,x_{23} 
\in \mathbb{R}\}.
\end{equation*}
Then, a linear transformation $\Upsilon:
H \rightarrow \mathbb{R}^{24}$ 
is defined as follows:
\begin{equation}\label{eq:trans1} 
\Upsilon((x_1, \ldots, x_{23}, 0)) 
:= 
\left (
\frac{x_1}{\sqrt{2}}, \frac{x_1}{\sqrt{2}}, 
\frac{x_2 - x_3}{\sqrt{2}}, \frac{x_2 + x_3}{\sqrt{2}}, \ldots, 
\frac{x_{22} - x_{23}}{\sqrt{2}}, \frac{x_{22} + x_{23}}{\sqrt{2}}
\right ) \in \mathbb{R}^{24}. 
\end{equation}
Since it can be easily verified that $\Upsilon$ preserves inner products, 
it follows that $\Upsilon$ is an orthogonal transformation.

\begin{theorem}\label{thm:SLL23}
With the orthogonal transformation $\Upsilon$, the point set $\{
\Upsilon((\boldsymbol{y},0))\mid 
\boldsymbol{y} \in 
(4\sqrt{15}v_{2,11})^{B_{23}}
\}$ is a subset of some 
shell of the shorter Leech lattice 
$O_{23}$.
\end{theorem}
\begin{proof}
For $v = \frac{1}{\sqrt{8}} (4, -4, 0, \ldots, 0) \in \Lambda_{24}$ with $\| v \|^2 = 4$, we set
\begin{align*}
v^{\perp} &= \left \{ \boldsymbol{x} = (x_1, \ldots, x_{24}) \in \Lambda_{24} \mid x_1 - x_2 = 0 \right \}.
\end{align*}
The orthogonal projection onto $v^{\perp}$ is defined by
\begin{align*}
\pi_{v^{\perp}}(\boldsymbol{t}) := \boldsymbol{t} - \frac{\langle \boldsymbol{t}, v \rangle}{\| v \|^2} v 
= \boldsymbol{t} - 
\frac{t_1 - t_2}{\sqrt{8}} v
\end{align*}
for $\boldsymbol{t} = (t_1, \ldots, t_{24}) \in \mathbb{R}^{24}$.
By the definition of the shorter Leech lattice $O_{23}$, 
we have
\begin{align*}
O_{23} = \left \{ 
\pi_{v^{\perp}}(\boldsymbol{x}) \mid \boldsymbol{x} \in \Lambda_{24},\; \langle \boldsymbol{x}, v \rangle \in 2 \mathbb{Z}
\right \} 
&\supset \left \{ 
\pi_{v^{\perp}}(\boldsymbol{x}) \mid \boldsymbol{x} = (x_1, \ldots, x_{24}) \in \Lambda_{24},\; x_1 = x_2
\right \} \\
&= \left \{ 
 (x_1, x_1, x_3, \ldots, x_{24})
 \in \Lambda_{24} \right \}.
\end{align*}

Our goal is to show that the point set $\{(\boldsymbol{y},0) \mid \boldsymbol{y} \in 
(4\sqrt{15}v_{2,11})^{B_{23}}
\} \subset H$ is transformed under the orthogonal transformation $\Upsilon$
into $\{(x_1,x_1,x_3,\ldots,x_{24}) \in \Lambda_{24}\} \; (\subset 
O_{23})$.

By noting 
\[ 
(4 \sqrt{15} v_{2, 11})^{B_{23}} = 
\frac{1}{\sqrt{8}} \big (16 \sqrt{2}, \underbrace{8 \sqrt{2}, \ldots, 8 \sqrt{2}}_{11 \text{ times }}, 0, \ldots, 0 \big )^{B_{23}}
\] 
and letting $(y_1, \ldots, y_{24}) \in 
\{(\boldsymbol{y},0) \mid \boldsymbol{y} \in (4\sqrt{15}v_{2,11})^{B_{23}}
\}$, 
we observe from equation (\ref{eq:trans1}) that each of $y_i/\sqrt{2}$ and $(y_i \pm y_{j})/\sqrt{2}$ belongs to the set 
$\frac{1}{\sqrt{8}} \{ 0, \pm 8,  \pm 16, \pm 24 \}$.
Hence, each resulting vector 
\[
\boldsymbol{z} \in \{
\Upsilon((\boldsymbol{y},0))\mid 
\boldsymbol{y} \in (4\sqrt{15}v_{2,11})^{B_{23}}
\}
\]
lies in $\frac{1}{\sqrt{8}} \mathbb{Z}^{24}$, and in fact we can write $\boldsymbol{z} = 4 \boldsymbol{z}'/\sqrt{8}$  
for some $\boldsymbol{z}' = (z'_1, z'_1, z'_3, \ldots, z'_{24})\in \mathbb{Z}^{24}$, with all $z'_i \in \{ 0, \pm 2, \pm 4, \pm 6 \}$.

Since all entries of $\boldsymbol{z}'$ are even integers, their sum is also even,  
and thus $\boldsymbol{z}$ satisfies the parity condition required for membership in $\Lambda_{24}$.  
Therefore, we conclude that $\boldsymbol{z} \in \Lambda_{24}$, 
which completes the proof.
\end{proof}

\section*{Acknowledgements}
This is a subsequent work of Bajnok~\cite{Bajnok2006} and Sawa and Xu~\cite{SX2013}, which was first started by the first and fourth authors. 
They would like to thank Yuan Xu and Tsuyoshi Miezaki for many valuable comments and suggestions to the former version of this paper.
We are grateful to the anonymous referees and the editor for their careful reading and constructive feedback, which have helped to improve the quality of the manuscript.

\appendix
\section{Proof of Proposition~\ref{prop:Hirao2024-2}}

In this section,
we give a proof of 
Proposition~\ref{prop:Hirao2024-2}.
Suppose $n\geq4$.
    We solve the following system of linear equations
\begin{equation*}
    \begin{bmatrix}
        1                     & 1                     \\
        f_{4}(v_{a_1, s_1})   & f_{4}(v_{a_2, s_2})   \\
        f_{6}(v_{a_1, s_1})   & f_{6}(v_{a_2, s_2})   \\
        f_{8,1}(v_{a_1, s_1}) & f_{8,1}(v_{a_2, s_2}) \\
        f_{8,2}(v_{a_1, s_1}) & f_{8,2}(v_{a_2, s_2})
    \end{bmatrix}
    \begin{bmatrix}
        \tilde{W}_1 \\ \tilde{W}_2
    \end{bmatrix}
    = \begin{bmatrix}
        1 \\ 0 \\ 0 \\ 0 \\ 0
    \end{bmatrix},
\end{equation*}
where $\tilde{W}_i = W_i |v_{a_i, s_i}^{B_n}| > 0 \; (i = 1, 2)$.
As in the proof of Proposition~\ref{prop:Hirao2024-1}, we divide the situation into five cases:
\[
    \begin{cases}
        (a) \; \text{$f_{4}(v_{a_1, s_1}) \neq f_{4}(v_{a_2, s_2})$}, \\
        (b) \; \text{$f_{4}(v_{a_1, s_1}) = f_{4}(v_{a_2, s_2}) \;, \;  f_{6}(v_{a_1, s_1}) \neq f_{6}(v_{a_2, s_2})$}, \\
        (c) \; \text{$f_{4}(v_{a_1, s_1}) = f_{4}(v_{a_2, s_2}) \;,\;  f_{6}(v_{a_1, s_1}) =  f_{6}(v_{a_2, s_2}) \;, \; f_{8, 1}(v_{a_1, s_1}) \neq  f_{8,1}(v_{a_2, s_2})$}, \\
        (d) \; \text{$f_{4}(v_{a_1, s_1}) = f_{4}(v_{a_2, s_2}) \;,\;  f_{6}(v_{a_1, s_1}) =  f_{6}(v_{a_2, s_2})$}, \\ \qquad \qquad 
        \text{$f_{8, 1}(v_{a_1, s_1}) =  f_{8,1}(v_{a_2, s_2})
        \;,\;  f_{8, 2}(v_{a_1, s_1}) \neq  f_{8,2}(v_{a_2, s_2})$}, \\
        (e) \; \text{$f_{4}(v_{a_1, s_1}) = f_{4}(v_{a_2, s_2}) \;,\; f_{6}(v_{a_1, s_1}) =  f_{6}(v_{a_2, s_2})$}, \\  \qquad \qquad 
        \text{$f_{8, 1}(v_{a_1, s_1}) =  f_{8,1}(v_{a_2, s_2})
                \;,\;  f_{8, 2}(v_{a_1, s_1}) =  f_{8,2}(v_{a_2, s_2})$}.
    \end{cases}
\]
In Case (a), the augmented coefficient matrix has rank $2$ and can be reduced to

\begin{align*}
\begin{bmatrix}
1 & 0 &  \frac{f_{4}(v_{a_2, s_2})}{f_{4}(v_{a_2, s_2}) - f_{4}(v_{a_1, s_1})}  \\
0 & 1 &  - \frac{f_{4}(v_{a_1, s_1})}{f_{4}(v_{a_2, s_2}) - f_{4}(v_{a_1, s_1})} \\
0 & 0  & \frac{G_{4, 6}(a_1, a_2, s_1, s_2)}{f_{4}(v_{a_2, s_2}) - f_{4}(v_{a,s})}  \\
0 & 0  & \frac{G_{4, 8,1}(a_1, a_2, s_1, s_2)}{f_{4}(v_{a_2, s_2}) - f_{4}(v_{a,s})}   \\
0 & 0  & \frac{G_{4, 8,2}(a_1, a_2, s_1, s_2)}{f_{4}(v_{a_2, s_2}) - f_{4}(v_{a,s})}  
\end{bmatrix}. 
\end{align*}
Thus, in Case (a), a weighted $9$-design of type (\ref{eq:invariant2orbit}) exists if and only if
\begin{align*}
\begin{cases}
\frac{f_{4}(v_{a_2, s_2})}{f_{4}(v_{a_2, s_2}) - f_{4}(v_{a_1, s_1})} > 0, \\
- \frac{f_{4}(v_{a_1, s_1})}{f_{4}(v_{a_2, s_2}) - f_{4}(v_{a_1, s_1})} > 0, \\
G_{4, 6}(a_1, a_2, s_1, s_2) = G_{4, 8,1}(a_1, a_2, s_1, s_2) = G_{4, 8,2}(a_1, a_2, s_1, s_2) = 0.
\end{cases} \notag 
\end{align*}
Since $f_{4}(v_{a_1, s_1}) \neq f_{4}(v_{a_2, s_2})$,\ this is equivalent to
\begin{align*}
    \begin{cases}
                  f_{4}(v_{a_1, s_1}) f_{4}(v_{a_2, s_2}) < 0,\\
                  G_{4, 6}(a_1, a_2, s_1, s_2) = G_{4, 8, 1}(a_1, a_2, s_1, s_2) = G_{4, 8, 2} (a_1, a_2, s_1, s_2) = 0.
    \end{cases}
\end{align*}

Similarly, in each of the remaining three cases (b), (c), (d), the augmented coefficient matrix has rank $2$ and obtain the desired equivalence case.

In the last case, 
Case (e),
the augmented coefficient matrix can be reduced to
\begin{align*}
\begin{bmatrix}
1 & 1 & 1 \\
0 & 0 & - f_{4}(v_{a_1, s_1})    \\
0 & 0 & - f_{6}(v_{a_1, s_1}) \\
0 & 0  & - f_{8,1}(v_{a_1, s_1})  \\
0 & 0  & - f_{8,2}(v_{a_1, s_1})
\end{bmatrix}. 
\end{align*}
Thus, a weighted $9$-design of type (\ref{eq:invariant2orbit}) exists if and only if
\begin{equation*}
\begin{cases}
f_{4}(v_{a_1, s_1}) = f_{4}(v_{a_2, s_2}) =
f_{6}(v_{a_1, s_1}) = f_{6}(v_{a_2, s_2}) = 0, \\
f_{8,1}(v_{a_1, s_1}) = f_{8,1}(v_{a_2, s_2}) =
f_{8,2}(v_{a_1, s_1}) = f_{8,2}(v_{a_2, s_2}) = 0.
\end{cases}
\end{equation*}

It remains to consider the case of $n = 3$.
In this case we may ignore $f_{8,2}$ and so reduce the size 
of the augmented coefficient matrix to $4 \times 3$ from $5 \times 3$.
 
\bibliographystyle{abbrv}
\bibliography{TTHS_refs}
\end{document}